\documentclass[12pt,reqno,twoside]{amsart}
\usepackage{geometry}                
\usepackage{graphicx}
\usepackage{amssymb}
\usepackage{epstopdf}

\usepackage{geometry}
\usepackage{graphicx}

\usepackage{booktabs} 
\usepackage{array} 
\usepackage{paralist} 
\usepackage{verbatim} 
\usepackage{subfig} 
\usepackage{tabularx}

\usepackage{amsmath,amsfonts,amsthm,mathrsfs,amssymb,cite}
\usepackage[usenames]{color}

\newtheorem{thm}{Theorem}[section]

\newtheorem{lem}{Lemma}[section]
\newtheorem{prop}{Proposition}[section]
\theoremstyle{definition}
\newtheorem{defn}{Definition}[section]
\theoremstyle{remark}

\newtheorem{rem}{Remark}[section]
\numberwithin{equation}{section}

\numberwithin{equation}{section}

\newcounter{saveeqn}

\newcommand{\eqnref}[1]{(\ref {#1})}

\newcommand{\Bz}{\mathbf{z}}

\newcommand{\ba}{\mathbf{a}}
\newcommand{\bE}{\mathbf{E}}

\newcommand{\bH}{\mathbf{H}}

\newcommand{\Bx}{\mathbf{x}}
\newcommand{\By}{\mathbf{y}}

\newcommand{\bV}{\mathbf{V}}

\newcommand{\BB}{\mathbf{B}}

\newcommand{\BW}{\mathbf{W}}

\newcommand{\GG}{\Gamma}
\newcommand{\Gs}{\sigma}

\newcommand{\GL}{\Lambda}

\newcommand{\tdx}{\tilde{\Bx}}
\newcommand{\tdy}{\tilde{\By}}

\newcommand{\Acal}{\mathcal{A}}
\newcommand{\Kcal}{\mathcal{K}}
\newcommand{\Lcal}{\mathcal{L}}
\newcommand{\Scal}{\mathcal{S}}

\newcommand{\Mcal}{\mathcal{M}}

\newcommand{\Ocal}{\mathcal{O}}

\newcommand{\ds}{\displaystyle}
\newcommand{\la}{\langle}
\newcommand{\ra}{\rangle}

\newcommand{\RR}{\mathbb{R}}

\newcommand{\p}{\partial}

\newcommand{\ti}{\tilde}

\newcommand{\beq}{\begin{equation}}
\newcommand{\eeq}{\end{equation}}

\DeclareMathAlphabet{\itbf}{OML}{cmm}{b}{it}

\title[Regularized Full and Partial EM Cloaks]{Full and partial cloaking in electromagnetic scattering}

\author{Youjun Deng}
\address{School of Mathematics and Statistics, Central South University, Changsha, Hunan, P. R. China.}
\email{youjundeng@csu.edu.cn, dengyijun\_001@163.com}

\author{Hongyu Liu}
\address{Department of Mathematics, Hong Kong Baptist University, Kowloon, Hong Kong SAR.}
\email{hongyu.liuip@gmail.com}

\author{Gunther Uhlmann}
\address{Department of Mathematics, University of Washington, Seattle, WA 98195, USA. \vspace*{-3mm}}
\address{\vspace*{-2mm} and}
\address{Institute for Advanced Study, Hong Kong University of Science and Technology, Hong Kong SAR.}
\email{gunther@math.washington.edu}

\date{} 
\begin{document}
\maketitle

\begin{abstract}

In this paper, we consider two regularized transformation-optics cloaking schemes for electromagnetic (EM) waves. Both schemes are based on the blowup construction with the generating sets being, respectively, a generic curve and a planar subset. We derive sharp asymptotic estimates in assessing the cloaking performances of the two constructions in terms of the regularization parameters and the geometries of the cloaking devices. The first construction yields an approximate full-cloak, whereas the second construction yields an approximate partial-cloak. Moreover, by incorporating properly chosen conducting layers, both cloaking constructions are capable of nearly cloaking arbitrary EM contents. This work complements the existing results in \cite{Ammari4,BL,BLZ} on approximate EM cloaks with the generating set being a singular point, and it also extends \cite{DLU15,LiLiuRonUhl} on regularized full and partial cloaks for acoustic waves governed by the Helmholtz system to the more challenging EM case governed by the full Maxwell system.

\medskip

\medskip



\end{abstract}

\section{Introduction}

\subsection{Background}
This paper concerns with invisibility cloaking of electromagnetic (EM) waves via the approach of transformation optics. This method which is based on the transformation properties of the optical material parameters and the invariance properties of the equations modeling the wave phenomena, was pioneered in \cite{GLU} for electrostatics. For Maxwell's system, the same transformation optics approach was developed in \cite{PenSchSmi}. In two dimensions, Leonhardt \cite{Leo} used the conformal mapping to cloak rays in the geometric optics approximation. The obtained cloaking materials are called {\it metamaterials}. The metamaterials proposed in \cite{GLU2,Leo,PenSchSmi} for the ideal cloaks are singular, and this has aroused great interest in the literature to deal with the singular structures. In \cite{GKLU3,LZ1}, the authors proposed to consider the finite-energy solutions from singularly weighted Sobolev spaces for the underlying singular PDEs, and both acoustic cloaking and EM cloaking were treated. In \cite{GKLUoe,GKLU_2,KOVW,KSVW,Liu}, the authors proposed to avoid the singular structures by incorporating regularization into the cloaking construction, and instead of ideal invisibility cloak, one considers approximate/near invisibility cloak. The latter approach has been further investigated in \cite{Ammari2,Ammari3,LiuSun} for the conductivity equation and Helmholtz system, modeling electric impedance tomography (EIT) and acoustic wave scattering, respectively; and in \cite{Ammari4,BL,BLZ} for the Maxwell system, modeling the EM wave scattering. For all of the above mentioned work on regularized approximate cloaks, the {\it generating set} is a singular point, and one always achieves the full-cloak; that is, the invisibility is attainable for detecting waves coming from every possible incident/impinging direction, and observations made at every possible angle. In a recent article \cite{LiLiuRonUhl}, the authors proposed to study regularized partial/customized cloaks for acoustic waves; that is, the invisibility is only attainable for limited/customized apertures of incidence and observation. The key idea is to properly choose the generating set for the blowup construction of the cloaking device. In \cite{LiLiuRonUhl}, the authors only proved qualitative convergence for the proposed partial/customized cloaking construction, and the corresponding result was further quantified in \cite{DLU15} by the authors of the current article.

In this paper, we shall extend \cite{DLU15,LiLiuRonUhl} on the regularized partial/customized cloaks for acoustic wave scattering to the case of electromagnetic wave scattering. We present two near-cloaking schemes with the generating sets being a generic curve or a planar subset, respectively. It is shown that the first scheme yields an approximate full-cloak, whose invisibility effect is attainable for the whole aperture of incidence and observation angles. The result obtained complements
\cite{Ammari4,BL,BLZ} on approximate full cloaks. However, the generating set for the blowup construction considered in \cite{Ammari4,BL,BLZ} is a singular point, whereas in this study, the generating set is a generic curve. The cloaking material in our full-cloaking scheme is less ``singular" than those in \cite{Ammari4,BL,BLZ}, but at the cost of losing some degree of accuracy on the invisibility approximation. The second scheme would yield an approximate partial-cloak with limited apertures of incidence and observation. We derive sharp asymptotic estimates in assessing the cloaking performances in terms of the regularization parameters and the geometries of the generating sets. The estimates are independent of the EM contents being cloaked, which means that the proposed cloaking schemes are capable of nearly cloaking arbitrary EM objects. Compared to \cite{Ammari4,BL,BLZ} on approximate full-cloaks, we need to deal with anisotropic geometries, whereas compared to \cite{DLU15,LiLiuRonUhl} on approximate partial-cloaks for acoustic scattering governed by the Helmholtz system, we need to tackle the more challenging Maxwell system. Finally, we refer the readers to \cite{CC,GKLU4,GKLU5,LiuUhl,U2} for surveys on the theoretical and experimental progress on transformation-optics cloaking in the literature.

\subsection{Mathematical formulation}

Consider a homogeneous space with the (normalized) EM medium parameters described by the electric permittivity $\varepsilon_0=\mathbf{I}_{3\times 3}$ and magnetic permeability $\mu_0=\mathbf{I}_{3\times 3}$. Here and also in what follows, $\mathbf{I}_{3\times 3}$ denotes the identity matrix in $\mathbb{R}^{3\times 3}$. For notational convenience, we also let $\sigma_0:=0\cdot \mathbf{I}_{3\times 3}$ denote the conductivity tensor of the homogeneous background space. We shall consider the invisibility cloaking in the homogeneous space described above. Following the spirt in \cite{Ammari4,BL,BLZ}, the proposed cloaking device is compactly supported in a bounded domain $\Omega$, and takes a three-layered structure. Let $\Omega_a\Subset\Omega_c\Subset\Omega$ be bounded domains such that $\Omega_a$, $\Omega_c\backslash\overline{\Omega}_a$ and $\Omega\backslash\overline{\Omega}_c$ are connected, and they represent, respectively, the cloaked region, conducting layer and cloaking layer of the proposed cloaking device. Let $\Gamma_0$ be a bounded open set in $\mathbb{R}^3$, and it shall be referred to as a {generating set} in the following. For $\delta\in\mathbb{R}_+$, we let $D_\delta$ denote an open neighborhood of $\Gamma_0$ such that $D_\delta\rightarrow \Gamma_0$ (in the sense of Hausdorff distance) as $\delta\rightarrow +0$. $D_\delta$ will be referred to as the {\it virtual domain}, and shall be specified below. Throughout, we assume that there exists a bi-Lipschitz and orientation-preserving mapping $F_\delta: \mathbb{R}^3\rightarrow\mathbb{R}^3$ such that
\begin{equation}\label{eq:transw}
F_\delta(D_{\delta/2})=\Omega_a,\ F_\delta(D_\delta\backslash\overline{D}_{\delta/2})=\Omega_c\backslash\overline{\Omega}_a,\ F_\delta(\Omega\backslash\overline{D}_\delta)=\Omega\backslash\overline{\Omega}_c\ \mbox{and}\ F_\delta|_{\mathbb{R}^3\backslash\Omega}=\mbox{Identity}.
\end{equation}
Next, we describe the EM medium parameter distributions $\{\mathbb{R}^3; \varepsilon,\mu,\sigma\}$ in the physical space containing the cloaking device, and $\{\mathbb{R}^3; \varepsilon_\delta,\mu_\delta,\sigma_\delta\}$ in the virtual space containing the virtual domain. The EM medium parameters are all assumed to be symmetric-positive-definite-matrix valued functions, and they characterize, respectively, the electric permittivity, magnetic permeability and electric conductivity. In what follows, $\{\mathbb{R}^3; \varepsilon,\mu,\sigma\}$ and $\{\mathbb{R}^3; \varepsilon_\delta,\mu_\delta,\sigma_\delta\}$ shall be referred to as, respectively, the physical and virtual scattering configurations. 
Let
\beq\label{eq:clkstruc1}
\{\RR^3; \varepsilon, \mu, \Gs\}=
\left\{
\begin{array}{ll}
\varepsilon_0, \mu_0, \Gs_0 & \mbox{in} \quad \RR^3\setminus\overline{\Omega}, \\
\varepsilon_c^*, \mu_c^*, \Gs_c^* & \mbox{in} \quad \Omega\setminus\overline{\Omega}_c, \\
\varepsilon_l^*, \mu_l^*, \Gs_l^* & \mbox{in} \quad \Omega_c\setminus\overline{\Omega}_{a}, \\
\varepsilon_a^*, \mu_a^*, \Gs_a^* & \mbox{in} \quad \Omega_{a},
\end{array}
\right.
\eeq
and
\beq\label{eq:struc}
\{\RR^3; \varepsilon_\delta, \mu_\delta, \Gs_\delta\}=
\left\{
\begin{array}{ll}
\varepsilon_0, \mu_0, \Gs_0 & \mbox{in} \quad \RR^3\setminus\overline{D}_\delta, \\
\varepsilon_l, \mu_l, \Gs_l & \mbox{in} \quad D_\delta\setminus\overline{D}_{\delta/2}, \\
\varepsilon_a, \mu_a, \Gs_a & \mbox{in} \quad D_{\delta/2}.
\end{array}
\right.
\eeq
The virtual and physical scattering configurations are connected by the so-called {\it push-forward} via the (blowup) transformation $F_\delta$ in \eqref{eq:transw}. To that end, we next introduce the push-forward of EM mediums. Let $m$ and $m_\delta$, respectively, denote the physical and virtual parameter tensors, where $m=\varepsilon, \mu$ or $\sigma$. Define the push-forward $(F_\delta)_*m_\delta$ as
\begin{equation}\label{eq:pfl}
m=(F_\delta)_* m_\delta:=\left(\frac{1}{\mbox{det}(DF_\delta)}(DF_\delta)\cdot m_\delta\cdot (DF_\delta)^{T}\right)\circ F_\delta^{-1},
\end{equation}
where $DF_\delta$ denotes the Jacobian matrix of the transformation $F_\delta$. Throughout the rest of our study, we assume that
\begin{equation}\label{eq:pfl2}
\{\mathbb{R}^3; \varepsilon,\mu,\sigma\}=(F_\delta)_*\{\mathbb{R}^3; \varepsilon_\delta,\mu_\delta,\sigma_\delta\}:=\{\mathbb{R}^3; (F_\delta)_*\varepsilon_\delta,(F_\delta)_*\mu_\delta,(F_\delta)_*\sigma_\delta\}.
\end{equation}

Next, we consider the time-harmonic EM wave scattering in the physical space. Let
\begin{equation}\label{eq:emp1}
\mathbf{E}^i(\mathbf{x}):=\mathbf{p} e^{i\omega \mathbf{x}\cdot \mathbf{d}},\quad \mathbf{H}^i:=\frac{1}{\omega}(\nabla\times \mathbf{E}^i)(\mathbf{x}), \ \ \mathbf{x}\in\mathbb{R}^3,
\end{equation}
where $\mathbf{p}\in\mathbb{R}^3\backslash\{0\}$, $\mathbf{d}\in\mathbb{S}^2$ and $\omega\in\mathbb{R}_+$. $(\mathbf{E}^i, \mathbf{H}^i)$ in \eqref{eq:emp1} is called a pair of EM plane waves with $\mathbf{E}^i$ the electric field and $\mathbf{H}^i$ the magnetic field. $\mathbf{p}$ is the polarization tensor, $\mathbf{d}$ is the incident direction and $\omega$ is the wavenumber of the plane waves $\mathbf{E}^i$ and $\mathbf{H}^i$. There always holds that
\begin{equation}\label{eq:pdo}
\mathbf{p}\perp\mathbf{d}, \quad \mbox{namely}\quad \mathbf{p}\cdot\mathbf{d}=0.
\end{equation}
$\mathbf{E}^i$ and $\mathbf{H}^i$ are entire solutions to the following Maxwell
equations
\begin{equation*}
 \left \{
 \begin{array}{ll}
\nabla\times{\bE^i}-i\omega\mu_0{\bH^i}=0  &\mbox{in} \quad \RR^3,\\
\nabla\times{\bH^i}+i\omega\varepsilon_0 {\bE^i}=0 &\mbox{in} \quad \RR^3,
 \end{array}
 \right .
 \end{equation*}
 The EM scattering in the physical space $\{\mathbb{R}^3;\varepsilon,\mu,\sigma\}$ due to the incident plane waves $(\mathbf{E}^i, \mathbf{H}^i)$ is described by the following Maxwell system
 \begin{equation}\label{eq:pss}
\left \{
 \begin{array}{ll}
\nabla\times{\bE}-i\omega\mu{\bH}=0  &\mbox{in} \quad \RR^3,\\
\nabla\times{\bH}+i\omega(\varepsilon+i\frac{\sigma}{\omega}) {\bE}=0 &\mbox{in} \quad \RR^3,
 \end{array}
 \right .
 \end{equation}
 subject to the Silver-M\"{u}ller radiation condition:
\beq\label{eq:radia2}
\lim_{\|\Bx\|\rightarrow\infty} \|\Bx\| \big( (\bH- \bH^i) \times\hat{\Bx}- (\bE-\bE^i)\big)=0,
\eeq
 where $\hat{\Bx} = \Bx/\|\Bx\|$ for $\mathbf{x}\in\mathbb{R}^3\backslash\{0\}$. We seek solutions $\mathbf{E}, \mathbf{H}\in H_{loc}(\mbox{curl}; \mathbb{R}^3)$ to \eqref{eq:pss}; see \cite{HaL,Lei1,Lei2,Ned} for the well-posedness of the scattering system \eqref{eq:pss}. Here and also in what follows, we shall often use the spaces
\[
H_{loc}(\mbox{curl}; X)=\{ U|_B\in H(\mbox{curl}; B);\ B\ \ \mbox{is any bounded subdomain of $X$} \}
\]
and
\[
H(\mbox{curl}; B)=\{ U\in (L^2(B))^3;\ \nabla\times U\in (L^2(B))^3 \}.
\]

It is known that the solution $\bE$ to \eqref{eq:pss} admits the following asymptotic expansion as $\|\Bx\|\rightarrow \infty$ (see, e.g., \cite{CK})
\beq\label{eq:farfieldptt1}
\bE(\mathbf{x})-\bE^i(\mathbf{x}) = \frac{e^{i \omega\|\Bx\|}}{\|\Bx\|}\mathbf{A}_\infty\Big(\frac{\Bx}{\|\Bx\|}; \bE^i\Big) +\Ocal\left(\frac{1}{\|\Bx\|^2}\right),
\eeq
where
\begin{equation}\label{eq:ffp}
\mathbf{A}_\infty(\hat{\mathbf{x}}; \mathbf{p}, \mathbf{d}):=\mathbf{A}_\infty\left(\frac{\Bx}{\|\Bx\|}; \bE^i\right)
\end{equation}
is known as the \emph{scattering amplitude} and $\hat{\mathbf{x}}$ denotes the direction of observation. It is readily verified that
\begin{equation}\label{eq:ppll}
\mathbf{A}_\infty(\hat{\mathbf{x}}; \mathbf{p}, \mathbf{d})=\|\mathbf{p}\|\mathbf{A}_\infty\left(\hat\Bx;\frac{\mathbf{p}}{\|\mathbf{p}\|}, \mathbf{d}\right).
\end{equation}
Therefore, without loss of generality and throughout the rest of our study, we shall assume that $\|\mathbf{p}\|=1$, namely, $\mathbf{p}\in\mathbb{S}^2$.
\begin{defn}\label{def:cloak}
Let $\Sigma_p\subset\mathbb{S}^2$, $\Sigma_d\subset\mathbb{S}^2$ and $\Sigma_{\hat{x}}\subset\mathbb{S}^2$. $\{\Omega; \varepsilon,\mu,\sigma\}$ is said to be a near/approximate-cloak if
\begin{equation}\label{eq:defcloak1}
\|\mathbf{A}_\infty(\hat{\mathbf{x}};\mathbf{p},\mathbf{d})\|\ll 1\ \ \mbox{for}\ \ \hat{\mathbf{x}}\in\Sigma_{\hat{x}},\ \mathbf{p}\in\Sigma_p,\ \mathbf{d}\in\Sigma_d.
\end{equation}
If $\Sigma_p=\Sigma_d=\Sigma_{\hat x}=\mathbb{S}^2$, then it is called an approximate full-cloak, otherwise it is called an approximate partial-cloak.
\end{defn}
According to Definition~\ref{def:cloak}, the cloaking layer $\{\Omega\backslash\overline{\Omega}_c; \varepsilon_c^*,\mu_c^*,\sigma_c^*\}$ together with the conducting layer $\{\Omega_c\backslash\overline{\Omega}_a; \varepsilon_l^*,\mu_l^*,\sigma_l^*\}$ makes the target EM object $\{\Omega_a; \varepsilon_a^*,\mu_a^*, \sigma_a^*\}$ nearly invisible to detecting waves \eqref{eq:emp1} with $\mathbf{d}\in \Sigma_d$ and $\mathbf{p}\in\Sigma_p$, and observation in the aperture $\Sigma_{\hat{x}}$. $\Sigma_d$ and ${\Sigma}_{\hat{x}}$ shall be referred to as, respectively, the apertures of incidence and observation of the partial-cloaking device. For practical considerations, throughout the current study, we assume that the cloaking device is not object-dependent; that is, the cloaked content $\{\Omega_a; \varepsilon_a^*,\mu_a^*, \sigma_a^*\}$ is {\it arbitrary but regular}, namely, $\varepsilon_a^*,\mu_a^*, \sigma_a^*$ are all arbitrary symmetric positive definite matrices. In Definition~\ref{def:cloak}, \eqref{eq:defcloak1} is rather qualitative, and in the subsequent study, we shall quantify the near-cloaking effect and derive sharp estimate in assessing the cloaking performance. To that end, the following theorem plays a critical role (cf. \cite{BL,BLZ}).

\begin{thm}\label{thm:transop}
Let $(\mathbf{E},\mathbf{H})\in H_{loc}(\mbox{\emph{curl}}; \mathbb{R}^3)^2$ be the (unique) pair of solutions to \eqref{eq:pss}. Define the \emph{pull-back fields} by
\[
\mathbf{E}_\delta=({F}_\delta)^* \mathbf{E}:=(D{F}_\delta)^{T}\mathbf{E}\circ {F}_{\delta},\quad \mathbf{H}_\delta=({F}_\delta)^* \mathbf{H}:=(D{F}_\delta)^{T}\mathbf{H}\circ {F}_{\delta}.
\]
Then the pull-back fields $(\mathbf{E}_\delta, \mathbf{H}_\delta)\in H_{loc}(\mbox{\emph{curl}};\mathbb{R}^3)^2$ satisfy the following Maxwell equations
\begin{equation}
\label{eq:sys1} \left\{
\begin{array}{ll}
 \nabla \times \bE_\delta - i \omega \mu_0 \bH_\delta=0 &  \mbox{in} \quad \RR^3\setminus \overline{D}_\delta, \\
 \nabla  \times \bH_\delta+i \omega \varepsilon_0 \bE_\delta=0 & \mbox{in} \quad \RR^3\setminus \overline{D}_\delta, \\
 \nabla \times \bE_\delta -i \omega \mu_\delta \bH_\delta=0 &  \mbox{in} \quad D_\delta, \\
 \nabla  \times \bH_\delta+ i \omega \Big(\varepsilon_\delta+i\frac{\Gs_\delta}{\omega}\Big) \bE_\delta =0 & \mbox{in} \quad D_\delta
\end{array}
\right.
\end{equation}  subject to the Silver-M\"{u}ller radiation condition:
\beq\label{eq:radia3}
\lim_{\|\Bx\|\rightarrow\infty} \|\Bx\| \big((\bH_\delta- \bH^i) \times\hat{\Bx}- (\bE_\delta-\bE^i)\big)=0,
\eeq
Particularly, since $F_\delta=\mbox{\emph{Identity}}$ in $\mathbb{R}^3\backslash\Omega$, one has that
\begin{equation}\label{eq:eqvp}
\mathbf{A}_\infty(\hat{\mathbf{x}}; \mathbf{E}^i)=\mathbf{A}^\delta_\infty(\hat{\mathbf{x}}; \mathbf{E}^i),\quad \hat{\mathbf{x}}\in\mathbb{S}^2,
\end{equation}
where $\mathbf{A}^\delta_\infty(\hat{\mathbf{x}}; \mathbf{E}^i)$ denotes the scattering amplitude corresponding to the Maxwell system \eqref{eq:sys1}.
\end{thm}

Hence, by Theorem~\ref{thm:transop}, in order to assess the cloaking performance of the cloaking device $\{\Omega;\varepsilon,\mu,\sigma\}$ in \eqref{eq:clkstruc1} associated with the physical scattering system \eqref{eq:pss}--\eqref{eq:radia2}, it suffices for us to investigate the virtual scattering system \eqref{eq:sys1}--\eqref{eq:radia3}. Here, it is noted for emphasis that by \eqref{eq:pfl2}, one has
\begin{equation}\label{eq:pfinv}
\{D_{\delta/2}; \varepsilon_a,\mu_a,\sigma_a\}=(F_\delta^{-1})_*\{\Omega_a;\varepsilon_a^*,\mu_a^*,\sigma_a^*\}.
\end{equation}
Since $\{\Omega_a;\varepsilon_a^*,\mu_a^*,\sigma_a^*\}$ is arbitrary and regular and $F_\delta$ is bi-Lipschitz and orientation-preserving, it is clear that $\{D_{\delta/2}; \varepsilon_a,\mu_a,\sigma_a\}$ is also arbitrary and regular.

In summarizing our discussion so far, in order to construct a near-cloaking device, one needs to firstly select a suitable generating set $\Gamma_0$ and the corresponding virtual domain $D_\delta$; and secondly the blowup transformation $F_\delta$ in \eqref{eq:transw}; and thirdly the virtual conducting layer $\{D_{\delta}\backslash\overline{D}_{\delta/2}; \varepsilon_l,\mu_l,\sigma_l\}$; and finally assess the cloaking performance by studying the virtual scattering system \eqref{eq:sys1}--\eqref{eq:radia3}. In this article, we shall be mainly concerned with extending the full- and partial-cloaking schemes proposed in \cite{DLU15,LiLiuRonUhl} for acoustic waves to the more challenging case with electromagnetic waves. There the generating sets are either a generic curve or a planar subset, and the blowup transformations are constructed via a concatenating technique. Hence, in the current study, we shall mainly focus on properly designing the suitable conducting layers and then assessing the cloaking performances by studying the corresponding virtual scattering system \eqref{eq:sys1}--\eqref{eq:radia3}.

 The rest of the paper is organized as follows. In Section 2, we collect some preliminary knowledge on boundary layer potentials, which shall be used throughout our study. Sections 3 and 4 are, respectively, devoted to the study of the full- and partial-cloaking schemes.

\section{Boundary layer potentials}\label{sect:3}

Our study shall heavily rely on the vectorial boundary integral operators for Maxwell's equations. In this section, we review some of the important properties of the vectorial boundary integral operators for the later use.

\subsection{Definitions}

Let $D$ be a bounded domain in $\mathbb{R}^3$ with a $C^3$-smooth boundary $\partial D$ and a connected complement $\mathbb{R}^3\backslash\overline{D}$. Let  $\nabla_{\p D}\cdot$ denote the surface divergence on $\partial D$ and
$\mathrm{H}^s(\partial D)$ be the usual Sobolev space of order $s\in\mathbb{R}$ on $\partial D$.
Let $\nu$ be the exterior unit normal vector to $\partial D$ and denote by $\mathrm{TH}^s(\p D):=\{\ba \in {\mathrm{H}^s(\p D)}^3; \nu\cdot \ba =0\}$, the space of
vectors tangential to $\p D$ which is a subset of $\mathrm{H}^s(\p D)^3$. We also introduce the function
space
\begin{align*}
\mathrm{TH}_{\mathrm{div}}^s(\p D):&=\Bigr\{ {\ba } \in \mathrm{TH}^s(\partial D);
\nabla_{\partial D}\cdot {\ba } \in \mathrm{H}^s(\partial D) \Bigr\},
\end{align*}
endowed with the norm
\begin{align*}
&\|{\ba }\|_{\mathrm{TH}_{\mathrm{div}}^s(\p D)}=\|{\ba }\|_{\mathrm{TH}^s(\p  D)}+\|\nabla_{\p  D}\cdot {\ba }\|_{\mathrm{H}^s(\p  D)}.
\end{align*}

Next, we recall that, for $\omega\in\mathbb{R}_+\cup\{0\}$, the fundamental outgoing solution $ G_\omega$ to the PDO $(\Delta+\omega^2)$ in $\mathbb{R}^3$ is given by
\begin{equation}\label{Gk} \ds G_\omega
(\Bx) = -\frac{e^{i\omega \|\Bx\|}}{4 \pi \|\Bx\|}.
 \end{equation}
In what follows, if $\omega=0$ we simply write $G_\omega$ as $G$.

For a density function $\ba  \in \mathrm{TH}_{\mathrm{div}}^s(\p D)$, we define the
vectorial single layer potential associated with the fundamental solution
$ G_\omega$ introduced in (\ref{Gk}) by
\begin{equation}\label{defA}
\ds\mathcal{A}_D^{\omega}[\ba ](\Bx) := \int_{\p  D}  G_\omega(\Bx-\By)
\ba (\By)\, d \sigma_y, \quad \Bx \in \mathbb{R}^3.
\end{equation}
For a scalar density $\varphi \in \mathrm{H}^{s}(\p  D)$, the single layer
potential is defined similarly by
\begin{equation}\label{defS}
\mathcal{S}_D^{\omega}[\varphi](\Bx) := \int_{\p  D}  G_\omega(\Bx-\By) \varphi(\By)\, d \sigma_y, \quad \Bx \in \mathbb{R}^3.
\end{equation}
The following boundary operator shall also be needed
\begin{equation}\label{defM}
\begin{aligned} \mathcal{M}_D^\omega:
\mathrm{L}_T^2 (\partial D)  & \longrightarrow \mathrm{L}_T^2 (\partial D)  \\
\ba  & \longrightarrow \mathcal{M}^\omega_D[\ba ](\Bx)= \mathrm{p.v. }\quad \nu_{\mathbf{x}}  \times \nabla \times \int_{\p  D}  G_\omega(\Bx,\By) \ba (\By)\, d\sigma_y,
\end{aligned}
\end{equation}
where $\mathrm{L}_T^2(\partial D):=\mathrm{TH}^0(\partial D)$, and p.v. signifies the Cauchy principle value.
In what follows, we denote by $\mathcal{A}_D$, $\mathcal{S}_D$ and $\mathcal{M}_D$ the operators $\mathcal{A}_D^0$, $\mathcal{S}_D^0$ and $\mathcal{M}_D^0$, respectively.

\subsection{Boundary integral  identities}
Here and throughout the rest of the paper, we make use of the following notation: for a function $u$ defined on $\mathbb{R}^3\backslash\partial D$, we denote
\[
u|_{\pm}(\Bx)=\lim_{\tau\rightarrow +0} u(\Bx\pm \tau\nu(\Bx)),\quad \Bx\in\partial D,
\]
and
\[
\frac{\partial u}{\partial\nu}\bigg|_\pm(\Bx)=\lim_{\tau\rightarrow +0}\langle \nabla_\Bx u(\Bx\pm \tau\nu(\Bx)),\nu(\Bx)\rangle,\quad \Bx\in\partial D,
\]
if the limits exist, where $\mathbf\nu$ is the unit outward normal vector to $\p D$.

It is known that the single layer potential $\Scal_D^\omega$
satisfies the trace formula (cf. \cite{CK,Ned})
\beq \label{eq:trace}
\frac{\p}{\p\nu}\Scal_{D}^\omega[\varphi] \Big|_{\pm} = \left(\pm \frac{1}{2}I+
(\Kcal_{D}^\omega)^*\right)[\varphi] \quad \mbox{on } \p D,
\eeq
where $(\Kcal_{D}^\omega)^*$ is the $L^2$-adjoint of $\Kcal_D^\omega$ and
$$
\Kcal_D^\omega [\ba ]:= \mbox{p.v.} \quad \int_{\p D} \frac{\p G_\omega(\Bx-\By)}{\p \nu(\By)} \varphi(\By)\, d\Gs_y , \quad \Bx\in \p D.
$$
The jump relations in the following proposition are also known (see \cite{CK,Ned}).
\begin{prop}\label{propjumpM}

Let $ \ba  \in \mathrm{TH}_{\mathrm{div}}^{-1/2}(\p  D)$. Then $\mathcal{A}_D^\omega[\ba ]$ is continuous on $\mathbb{R}^3$ and its \emph{curl} satisfies the following jump formula,
\begin{equation}\label{jumpM}
\nu \times \nabla \times \mathcal{A}_D^\omega[\ba ]\big\vert_\pm = \mp \frac{\ba }{2} + \mathcal{M}_D^\omega[\ba ] \quad \mbox{ on } \p  D,
\end{equation}
where
\begin{equation*}
 \nu(\mathbf{x}) \times \nabla \times \mathcal{A}_D^\omega[\ba ]\big\vert_\pm (\mathbf{x})= \lim_{t\rightarrow +0} \nu(\mathbf{x}) \times \nabla \times \mathcal{A}_D^\omega[\ba ] (\mathbf{x}\pm t \nu(\mathbf{x})),\quad \forall \mathbf{x}\in \p  D,
\end{equation*}
\end{prop}
Equipped with the above knowledge, the solution pair $(\bE_\delta, \bH_\delta)$ in $\RR^3\setminus\overline{D}_\delta$ to \eqnref{eq:sys1} can be represented using the following integral ansatz,
\begin{align}
&\bE_\delta(\Bx)= \bE^i(\Bx) + \nabla_{\mathbf{x}}\times \Acal_{D_\delta}^{\omega}[\ba](\Bx), & \Bx\in \RR^3\setminus\overline{D}_\delta,\label{eq:vis1}\\
&\bH_\delta(\Bx)=\frac{1}{i\omega}\nabla_{\mathbf{x}}\times\bE_\delta(\Bx)=\bH^i(\Bx) + \frac{1}{i\omega}\nabla_{\mathbf{x}}\times\nabla_{\mathbf{x}}
\times \Acal_{D_\delta}^{\omega}[\ba](\Bx), & \Bx\in \RR^3\setminus\overline{D}_\delta, \label{eq:vis2}
\end{align}
where by \eqnref{jumpM} the vectorial density function $\ba\in {\mathrm{TH}}^{-1/2}_{\mathrm{div}}(\p D_\delta)$ satisfies
\beq\label{eq:trans1}
\Big(-\frac{I}{2}+\Mcal_{D_\delta}^{\omega}\Big)[\ba](\Bx)=\mathbf\nu\times (\bE_\delta-\bE^i)(\Bx)\Big|_+, \quad \By\in \p D_\delta.
\eeq

\section{Regularized full-cloaking of EM waves}\label{sect:4}

In this section, we consider a regularized full-cloaking scheme of the EM waves by taking the generating set to be a generic curve. As discussed in the introduction, this scheme was considered in our earlier work \cite{DLU15} for acoustic waves. For self-containedness, we
briefly discuss the generating set $\Gamma_0$ and the virtual domain $D_\delta$ for the proposed cloaking scheme in the sequel, which can also be found in \cite{DLU15}. Let $\GG_0$ be a smooth simple and non-closed curve in $\RR^3$ with two endpoints, denoted by $P_0$ and $Q_0$, respectively. Denote by $N(\Bx)$ the normal plane of the curve $\GG_0$ at $\Bx\in\GG_0$. We note that $N(P_0)$ and $N(Q_0)$ are, respectively, defined by the left and right limits along $\Gamma_0$. Let $q\in\mathbb{R}_+$. For any $\Bx\in\GG_0$, we let $\mathscr{S}_q(\Bx)$ denote the disk lying on $N(\Bx)$, centered at $\Bx$ and of radius $q$. It is assumed that there exists $q_0\in\mathbb{R}_+$ such that when $q\leq q_0$, $\mathscr{S}_q(\Bx)$ intersects $\GG_0$ only at $\Bx$. We let $D_q^f$ be given as
\begin{equation}\label{eq:drf}
D_q^f:=\mathscr{S}_q(\Bx)\times \Gamma_0(\Bx),\ \Bx\in\overline{\Gamma}_0,
\end{equation}
where $\Gamma_0$ is identified with its parametric representation $\Gamma_0(\Bx)$; see Fig.~\ref{fig1} for a schematic illustration. Clearly, the facade of $D_q^f$, denoted by $S_q^f$ and parallel to $\GG_0$, is given by
\beq\label{eq:facade}
S_q^f:=\{\Bx+q\cdot \mathbf{n}(\Bx); \Bx\in \GG_0, \mathbf{n}(\Bx)\in N(\Bx) \cap \mathbb{S}^2\},
\eeq
and the two end-surfaces of $D_q^f$ are the two disks $\mathscr{S}_q(P_0)$ and $\mathscr{S}_q(Q_0)$. Let $D_{q_0}^a$ and $D_{q_0}^b$ be two simply connected sets with $\partial D_{q_0}^a=S_{q_0}^a\cup \mathscr{S}_{q_0}(P_0)$ and $\partial D_{q_0}^b=S_{q_0}^b\cup\mathscr{S}_{q_0}(Q_0)$, such that $S_{q_0}:=S^f_{q_0}\cup S^b_{q_0} \cup S^a_{q_0}$ is a $C^3$-smooth boundary of the domain $D_{q_0}:=D_{q_0}^a\cup D_{q_0}^f\cup D_{q_0}^b$. For $0<q<q_0$, we set
\[
D_q^a:=\frac{q}{q_0}(D_{q_0}^a-P_0)+P_0=\left\{\frac{q}{q_0}\cdot(\Bx-P_0)+P_0; \Bx\in D_{q_0}^a\right\},
\]
and similarly, $D_q^b:={q}/{q_0}\cdot(D_{q_0}^b-Q_0)+Q_0$. Let $S_q^a$ and $S_q^b$, respectively, denote the boundaries of $D_q^a$ and $D_q^b$ excluding $\mathscr{P}_q^a$ and $\mathscr{S}_q^b$. Now, we set $D_q:=D_q^a\cup D_q^f\cup D_q^b$, and $S_q:=S^f_q\cup S^b_q \cup S^a_q=\partial D_q$.
\begin{figure}
\begin{center}
  \includegraphics[width=4.0in,height=2.0in]{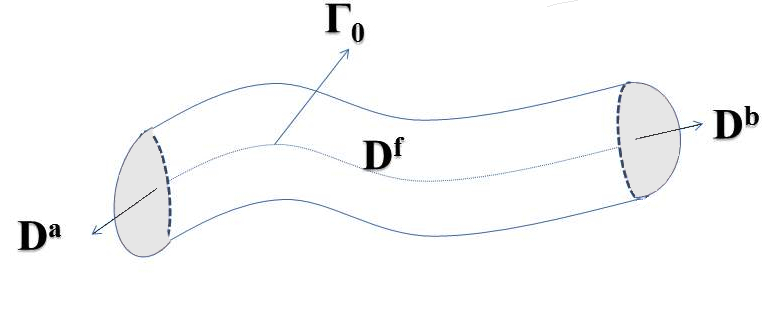}
  \end{center}
  \caption{Schematic illustration of the domain $D_q$ for the regularized full-cloak. \label{fig1}}
\end{figure}

Henceforth, we we let $\delta\in\mathbb{R}_+$ be the asymptotically small regularization parameter and let $D_\delta$ denote the virtual domain used for the blowup construction of the cloaking device. We also let
\beq\label{eq:rgn1}
S_\delta:=S^f_\delta \cup S^b_\delta \cup S^a_\delta,
\eeq
denote the boundary surface of the virtual domain $D_\delta$. Without loss of generality, we assume that $q_0\equiv 1$. We shall drop the dependence on $q$ if one takes $q=1$. For example, $D$ and $S$ denote, respectively, $D_q$ and $S_q$ with $q=1$. It is remarked that in all of our subsequent arguments, $D$ can always be replaced by $D_{\tau_0}$ with $0<\tau_0\leq q_0$ being a fixed number.

In what follows, if we utilize $\Bz$ to denote the space variable on $\GG_0$ , then for every $\By\in D_q^f$, we define a new variable $\Bz_{y}\in \GG_0$ which is the projection of $\By$ onto $\GG_0$.
Meanwhile, if $\By$ belongs to $D_q^a$ (respectively $D_q^b$), then $\Bz_y$ is defined to be $P_0$ (respectively $Q_0$).
Henceforth, we let $\xi$ denote the arc-length parameter of $\GG_0$ and $\theta$,
which ranges from $0$ to $2\pi$, be the angle of the point on $N(\xi)$ with respect to the central point $\Bx(\xi)\in \GG_0$. Moreover, we assume that if $\theta=0$, then the corresponding points are those lying on the line that connects $\GG_0(\xi)$ to $\GG_1(\xi)$, where
$\GG_1$ is defined to be
$$
\GG_1:=\{\Bx + \mathbf{n}_1(\Bx); \Bx \in \GG_0, \mathbf{n}_1(\Bx)\in N(\Bx)\cap \mathbb{S}^2-\mbox{a fixed vector for a given $\Bx$}\}.
$$

With the above preparation, we introduce a blowup transformation which maps $\By\in\overline{D}_\delta$ to $\tilde\By\in \overline{D}$ as follows
\begin{equation}\label{eq:A}
A(\By)=\tilde\By:=\frac{1}{\delta}(\By-\Bz_y)+\Bz_y,\quad \By\in D_\delta^f,
\end{equation}
whereas
\beq\label{eq:Ares1}
A(\By)=\tilde\By:= \left\{
\begin{array}{ll}
\frac{\By-P_0}{\delta}+P_0, & \By \in D_\delta^a, \\
\frac{\By-Q_0}{\delta}+Q_0, & \By \in D_\delta^b.
\end{array}
\right.
\eeq
Next, we present the crucial design of the lossy layer in \eqnref{eq:struc}.
Define the Jacobian matrix $\BB$ by
\begin{equation}\label{eq:jacob1}
\BB(\By)=\nabla_\By A(\By),\quad \By\in\overline{D}_\delta.
\end{equation}
Set the material parameters $\varepsilon_\delta$, $\mu_\delta$ and $\Gs_\delta$
in the lossy layer $D_\delta\setminus\overline{D}_{\delta/2}$ to be
\beq\label{eq:loss1}
\begin{array}{ll}
\varepsilon_\delta(\Bx)=\varepsilon_l(\Bx):=\delta^{r}|\BB|\BB^{-1}, &
\mu_\delta(\Bx)=\mu_l(\Bx):=\delta^{s}|\BB|\BB^{-1}, \\
\Gs_\delta(\Bx)=\Gs_l(\Bx):=\delta^{t}|\BB|\BB^{-1}, & \mbox{for} \quad \Bx\in D_\delta\setminus\overline{D}_{\delta/2},
\end{array}
\eeq
where $r$, $s$ and $t$ are all real numbers and $|\cdot|$ stands for the determinant when related to a square matrix.


We are now in a position to present the main theorem on the approximate full-cloak constructed by using $D_\delta$ described above as the virtual domain.
\begin{thm}\label{th:main1}
Let $D_\delta$ be as described above with $\partial D_\delta=S_\delta$ defined in \eqnref{eq:rgn1}.
Let $(\bE_\delta,\bH_\delta)$ be the pair of solutions to \eqnref{eq:sys1}, with
$\{\Omega; \varepsilon_\delta, \mu_\delta, \Gs_\delta\}\subset \{\RR^3; \varepsilon_\delta, \mu_\delta, \Gs_\delta\} $
defined in \eqnref{eq:struc}, and $\{D_\delta\backslash\overline{D}_{\delta/2}; \varepsilon_\delta, \mu_\delta, \Gs_\delta\}$ given in \eqnref{eq:loss1}. Define
$$
\beta=\min\{1, -1+r+s, -1+t+s\}, \quad \beta'=\min\{1, -2+r+s, -2+t+s\}.
$$
If $r$, $s$ and $t$ are chosen such that $\beta'-t/2\geq 0$, then there exists $\delta_0\in\mathbb{R}_+$ such that when $\delta<\delta_0$,
\beq\label{eq:ncest1}
\|\mathbf{A}^{\delta}_\infty(\hat{\Bx}; \mathbf{p}, \mathbf{d})\|\leq C (\delta^{\beta-t/2+1}+\delta^2)
\eeq
where $C$ is a positive constant depending on $\omega$ and $D$,
but independent of $\varepsilon_a$, $\mu_a$, $\Gs_a$ and $\hat{\Bx}$, $\mathbf{p}$, $\mathbf{d}$.
\end{thm}

\begin{rem}\label{lem:physical1}
Following our earlier discussion, one can immediately infer by Theorems~\ref{thm:transop} and \ref{th:main1} that the push-forwarded structure
in \eqnref{eq:clkstruc1},
$$\{\Omega; \varepsilon, \mu, \Gs\}=(F_\delta)_*\{\Omega; \varepsilon_\delta, \mu_\delta, \Gs_\delta\}$$
produces an approximate full-cloaking device within at least $\delta$-accuracy to the ideal cloak.
Indeed, if we set $s=2$ and $r=t=0$ then one has $\beta=1$, $\beta'=0$ and the accuracy
of the ideal cloak will be $\delta^2$, which is the highest accuracy that one can obtain for such a construction. Particularly, it is emphasize that in \eqref{eq:ncest1}, the estimate is independent of $\varepsilon_a$, $\mu_a$, $\Gs_a$, and this means that the cloaked content $\{\Omega_a;\varepsilon_a^*,\mu_a^*,\sigma_a^*\}$ in \eqnref{eq:clkstruc1} can be arbitrary but regular. Finally, as remarked earlier, we refer to \cite{LiLiuRonUhl} for the construction of the blowup transformation $F_\delta$, which we always assume the existence in the current study.
\end{rem}

The subsequent three subsections are devoted to the proof of Theorem~\ref{th:main1}. For our later use, we first derive some critical lemmas.
\subsection{Auxiliary lemmas}
In this subsection we present some auxiliary lemmas that are essential for our analysis of the far field estimates.
To begin with, we show the following properties of the blowup transformation defined in \eqnref{eq:A}
\begin{lem}[Lemma 4.1 in \cite{DLU15}]\label{le:2}
Let $A$ be the transformation introduced in \eqref{eq:A} and \eqref{eq:Ares1} which maps the region $\overline{D}_\delta$ to $\overline{D}$.
Let $\BB(\By)$ be the corresponding Jacobian matrix of $A(\By)$ given by \eqnref{eq:jacob1}.
Then we have
\beq\label{eq:jocob11}
\BB(\By)=\left\{
\begin{array}{ll}
\frac{1}{\delta}\mathbf{I}_{3\times 3}-\Big(\frac{1}{\delta}-1\Big)\Bz_y'(\xi)\Bz_y'(\xi)^T, & \By\in D_\delta^f, \\
\frac{1}{\delta}\mathbf{I}_{3\times 3}, &\By\in D_\delta^a\cup D_\delta^b,
\end{array}
\right.
\eeq
where the superscript $T$ denotes the transpose of a vector or a matrix. Furthermore,
\beq\label{eq:le22}
 \BB(\By)\mathbf\nu_{\mathbf{y}}= \frac{1}{\delta}\mathbf\nu_{\mathbf{y}},\quad \By\in\partial D_\delta,\\
\eeq
where $\nu_{\mathbf{y}}$ stands for the unit outward normal vector to $\p D_\delta$ at  $\By\in\partial D_\delta$.
\end{lem}
\begin{rem}
In view of the Jacobian matrix form \eqnref{eq:jocob11}, one can also find that the eigenvalues of $\BB(\Bx)$, $\Bx\in D_\delta$ are
either $1$ or $1/\delta$. Hence for any vector field $\bV\in \RR^3$, there holds
\beq\label{eq:eigenB1}
\|\bV\|^2\leq \la \BB(\Bx)\bV, \bV \ra \leq \delta^{-1} \|\bV\|^2
\eeq
uniformly for $\Bx\in D_\delta$. It can also be easily seen from \eqnref{eq:jocob11} that
\beq\label{eq:jacobdet1}
|\BB(\By)|=\delta^{-2}, \quad \By \in D_\delta^f.
\eeq
\end{rem}
For the sake of simplicity, we define
\beq\label{eq:def1}
\bE_\delta^+:=\bE_\delta-\bE^i, \quad \bH_\delta^+:=\bH_\delta-\bH^i, \quad \mbox{in} \quad \RR^3\setminus \overline{D}_\delta.
\eeq
Furthermore, we introduce the following notations
\beq\label{eq:def2}
\ti\bE(\ti\Bx):=\bE(A^{-1}(\ti\Bx))=\bE(\Bx), \quad \ti\bH(\ti\Bx):=\bH(A^{-1}(\ti\Bx))=\bH(\Bx),
\eeq
and define the corresponding fields after change of variables by
\beq\label{eq:def3}
\hat{\bE}(\ti\Bx):=((\ti\BB^T)^{-1}\ti\bE)(\ti\Bx), \quad \hat{\bH}(\ti\Bx):=((\ti\BB^T)^{-1}\ti\bH)(\tdx),
\eeq
where
$$
\tilde{\BB}(\tdx):=\BB(A^{-1}(\tdx))=\BB(\Bx).
$$
We mention that sometimes we write $\BB$ and $\ti\BB$ in the sequel and omit their dependences for simplicity.
The following lemma is of critical importance for our subsequent analysis.
\begin{lem}[Corollary 3.58 in \cite{Monk03}]\label{le:changerela1}
Let $\ti\Bx=A(\Bx)$ with $\BB(\Bx)=\nabla A(\Bx)$.
Then for the bounded domain $D_\delta$ and any vector field $\bV\in H(\mbox{\emph{curl}}; D_\delta)$,
$$\hat{\bV}(\tdx):=(\ti\BB^T)^{-1} \ti \bV(\ti\Bx),\quad \ti \bV(\ti\Bx):=\bV(\Bx), \quad \Bx\in D_\delta, $$
there hold the following identities
\beq\label{eq:ppp41}
|\ti\BB|^{-1}\ti\BB(\nabla\times \bV) (A^{-1}(\tdx))=\nabla_{\ti \Bx}\times \hat{\bV}(\tdx), \quad
\eeq
and
\beq\label{eq:pp411}
\int_{\p D_\delta}(\nu_{\Bx}\times \bV)\cdot \BW d\Gs_x= \int_{\p D}(\nu_{\tdx}\times \hat{\bV})\cdot \hat{\BW} d\Gs_{\tilde{x}}
\eeq
where $\BW\in H(\mbox{\emph{curl}}; D_\delta)$ and $\hat{\BW}(\tdx):=(\ti\BB^T)^{-1}\ti{\BW}(\tdx):=(\BB^T)^{-1}\BW(\Bx)$.
\end{lem}
Note that
\beq\label{eq:areapro411}
d\Gs_y=\left\{
\begin{array}{ll}
\delta\, d\Gs_{\tilde{y}}, & \By\in S_\delta^f,\\
\delta^2\, d\Gs_{\tilde{y}}, & \By\in S_\delta^b\cup S_\delta^a,
\end{array}
\right.
\eeq
with which one can show that 
\begin{lem}\label{le:changerela2}
Let $D^c$, $c\in \{f, a, b\}$ be defined at the beginning of this section.
Let $\bV$ and $\hat{\bV}$, $\BB$ be similarly defined as those in Lemma \ref{le:changerela1}.
Then for any $\BW\in H(\mbox{\emph{curl}};D)$, one has
\begin{align}
\int_{\p D^f}\nu_{\tdx}\times \tilde\bV\cdot \BW(\tdx)\, d\Gs_{\ti x} = & \delta^{-1} \int_{\p D^f} \ti \nu_{\tdx}\times \hat{\bV}
\cdot \hat\BW(\tdx)\, d\Gs_{\ti x}, \label{eq:imp411}\\
\int_{\p D^c}\nu_{\tdx}\times \ti\bV\cdot \BW(\tdx)\, d\Gs_{\ti x} = & \delta^{-2} \int_{\p D^c} \nu_{\tdx}\times \hat{\bV} \cdot
 \hat\BW(\tdx)\, d\Gs_{\ti x} \quad c\in \{a, b\}, \label{eq:imp412}
\end{align}
where $\hat{\BW}(\tdx):=(\ti\BB^T)^{-1}\BW(\tdx)$.
\end{lem}
\begin{proof}
The proof follows directly from \eqnref{eq:pp411}, \eqnref{eq:areapro411} by using change of variables in the corresponding integrals. 
\end{proof}

\begin{lem}\label{le:1}
Suppose $\tilde\ba (\tdx)=\ba (\Bx)$ for $\Bx\in\partial D_\delta$ and $\tilde\Bx\in\partial D$.
Define
$$
{\Mcal}_{S^f}^{\omega}[\tilde\ba ](\tdx):= \mathrm{p.v.}\quad \mathbf{\nu}_{\ti{\mathbf{x}}}\times \nabla_{\mathbf{z}_{\ti{x}}}\times \int_{S^f}{G_\omega}(\Bz_{\ti{x}}-\Bz_{\ti{y}})\ti\ba (\tdy)\, d\Gs_{\tilde{y}},$$
and
$$
\Mcal_{S^c}[\tilde\ba ](\tdx):=\mathrm{p.v.}\quad \mathbf{\nu}_{\ti{\mathbf{x}}}\times \nabla_{\ti{\mathbf{x}}}\times \int_{S^c} G_\omega(\tdx-\tdy)\ti\ba (\tdy)\, d\Gs_{\ti{y}}.
$$
Define
$$\Lcal_{D_\delta}^{\omega}(\Bx):=\nu_{\mathbf{x}}\times\nabla_{\mathbf{x}}\times\nabla_{\mathbf{x}}
\times \Acal_{D_\delta}^{\omega}[\ba](\Bx).$$
Then there hold the following results
\beq\label{eq:arg1}
\Mcal_{D_\delta}^\omega[\ba ](\Bx)=
\left\{
\begin{array}{ll}
\delta \Mcal_{S^f}^\omega[\tilde\ba ](\tdx)+ \Ocal(\delta^2\|\tilde\ba \|_{\mathrm{TH}^{-1/2}(\p D)}), & \Bx \in S_\delta^f, \\
\Mcal_{S^c}[\tilde\ba ](\tdx)+ \Ocal(\delta\|\tilde\ba \|_{\mathrm{TH}^{-1/2}(\p D)}), & \Bx \in S_\delta^c,
\end{array}
\right.
\eeq
and
\beq\label{eq:arg2}
\Lcal_{D_\delta}^{\omega}[\ba](\Bx)=
\left\{
\begin{array}{ll}
\delta\Lcal_{S^f}^\omega[\tilde\ba ](\tdx)+ \Ocal(\delta^2\|\tilde\ba \|_{\mathrm{TH}^{-1/2}(\p D)}), & \Bx \in S_\delta^f, \\
{\frac{1}{\delta} \nu_{\tilde{\mathbf{x}}}\times\nabla_{\tilde{\mathbf{x}}}\times\nabla_{\tilde{\mathbf{x}}}\times\Acal_{S^c}[\ti\ba](\tdx)+ \Ocal(\delta\|\tilde\ba \|_{\mathrm{TH}^{-1/2}(\p D)})}, & \Bx \in S_\delta^c,
\end{array}
\right.
\eeq
for $c\in\{a, b\}$, where
$$
\Lcal_{S^f}^{\omega}[\ti\ba](\Bx):=\nu_{\ti{\mathbf{x}}}\times\nabla_{\mathbf{z}_{\ti x}}\times\nabla_{\mathbf{z}_{\ti x}}
\times \int_{S^f}G_\omega(\Bz_{\ti{x}}-\Bz_{\ti{y}})\ti\ba (\tdy)\, d\Gs_{\ti{y}},
$$
and $\Acal_{S^c}$ is given in \eqref{defA} by replacing $\partial D$ and $\omega$ with $S^c$ and $0$, respectively.
\end{lem}

\begin{proof}
For
$\Bx, \By\in D_\delta$, one has
$$
\Bx-\By=(\Bx-\Bz_x)-(\By-\Bz_y)+\Bz_x-\Bz_y=\delta((\tdx-\Bz_{\tilde{x}})-(\tdy-\Bz_{\tilde{y}}))+(\Bz_{\tilde{x}}-\Bz_{\tilde{y}}).
$$
Hence, we have the following expansion for $\Bx\in S_\delta^f$,
$$
\|\Bx-\By\|=\|\Bz_{\tilde{x}}-\Bz_{\tilde{y}}\|+\delta \big\la(\tdx-\Bz_{\tilde{x}})-(\tdy-\Bz_{\tilde{y}}), \frac{\Bz_{\tilde{x}}-\Bz_{\tilde{y}}}{\|\Bz_{\tilde{x}}-\Bz_{\tilde{y}}\|}\big\ra + \Ocal(\delta^2).
$$
Similarly
\begin{align*}
& \big\la \Bx-\By, \mathbf\nu_{\mathbf{x}}\big\ra = \big\la \Bz_{\tilde{x}}-\Bz_{\tilde{y}}, \mathbf\nu_{\mathbf{x}}\big\ra +\delta \big\la (\tdx-\Bz_{\tilde{x}})-(\tdy-\Bz_{\tilde{y}}), \mathbf\nu_{\mathbf{x}}\big\ra,\\
&\|\Bx-\By\|^{-1}=\|\Bz_{\tilde{x}}-\Bz_{\tilde{y}}\|^{-1}-\delta \big\la(\tdx-\Bz_{\tilde{x}})-(\tdy-\Bz_{\tilde{y}}), \frac{\Bz_{\tilde{x}}-\Bz_{\tilde{y}}}{\|\Bz_{\tilde{x}}-\Bz_{\tilde{y}}\|^3}\big\ra + \Ocal(\delta^2),\\
&e^{i\omega \|\Bx-\By\|}=e^{i\omega\|\Bz_{\tilde{x}}-\Bz_{\tilde{y}}\|}\Big(1+i\omega\delta\big\la (\tdx-\Bz_{\tilde{x}})-(\tdy-\Bz_{\tilde{y}}), \frac{\Bz_{\tilde{x}}-\Bz_{\tilde{y}}}{\|\Bz_{\tilde{x}}-\Bz_{\tilde{y}}\|}\big\ra\Big)
+\Ocal(\delta^2).
\end{align*}
With those expansions at hand, we proceed to compute for $\Bx\in S_\delta^f$,
\begin{align*}
\nabla_{\mathbf{x}} G_\omega(\Bx-\By)&=\frac{(\Bx-\By)e^{i\omega\|\Bx-\By\|}}{4\pi\|\Bx-\By\|^2}\left(\frac{1}{\|\Bx-\By\|}-i\omega\right)\\
&=
\frac{(\Bz_{\tilde{x}}-\Bz_{\tilde{y}})e^{i\omega\|\Bz_{\tilde{x}}-\Bz_{\tilde{y}}\|}}{4\pi\|\Bz_{\tilde{x}}-\Bz_{\tilde{y}}\|^2}
\left(\frac{1}{\|\Bz_{\tilde{x}}-\Bz_{\tilde{y}}\|}-i\omega\right)+\Ocal(\delta)\\
&=\nabla_{\mathbf{z}_{\tilde{x}}} G_\omega(\Bz_{\tilde{x}}-\Bz_{\tilde{y}})+\Ocal(\delta),
\end{align*}
and by using vector calculus identities,
\begin{align*}
&\Mcal_{D_\delta}^\omega[\ba ](\Bx)=\mathbf{\nu}_{\mathbf{x}}\times \nabla_{\mathbf{x}}\times\int_{\p D_\delta} G_\omega(\Bx-\By)\ba (\By)\, d\Gs_y \\
=&\int_{\p D_\delta} \nabla_{\mathbf{x}} G_\omega(\Bx-\By)\mathbf{\nu}_{\mathbf{x}}\cdot\ba (\By)\, d\Gs_y
-\int_{\p D_\delta}\mathbf{\nu}_{\mathbf{x}}\cdot\nabla_{\mathbf{x}} G_\omega(\Bx-\By)\ba (\By)\, d\Gs_y\\
=&\delta\int_{S^f} \nabla_{\mathbf{z}_{\tilde{{x}}}} G_\omega(\Bz_{\tilde{x}}-\Bz_{\tilde{y}})\mathbf{\nu}_{\tilde{\mathbf{x}}}\cdot\tilde\ba (\tdy)\, d\Gs_{\tilde{y}}
-\delta\int_{S^f}\mathbf{\nu}_{\tilde{\mathbf{x}}}\cdot\nabla_{\mathbf{z}_{\tilde{x}}} G_\omega(\Bz_{\tilde{x}}-\Bz_{\tilde{y}})\tilde\ba (\tdy)\, d\Gs_{\tilde{y}}
\\&+\Ocal(\delta^2\|\tilde\ba \|_{\mathrm{TH}^{-1/2}(\p D)})
=\delta\Mcal_{S^f}^\omega[\tilde\ba ](\tdx)+ \Ocal(\delta^2\|\tilde\ba \|_{\mathrm{TH}^{-1/2}(\p D)}),
\end{align*}
which proves the case of \eqnref{eq:arg1} for $\Bx\in S_\delta^f$. Next, note that if $\Bx\in S_\delta^c$, $c\in \{a ,b \}$, then
$$
\Bz_{\tilde{x}}-\Bz_{\tilde{y}}=0 \quad \mbox{for} \quad \By\in S_\delta^c,
$$
and
$$
\|\Bx-\By\|=\delta\|\tdx-\tdy\|,\quad \frac{\la\Bx-\By,\mathbf\nu_{\mathbf{x}}\ra}{\|\Bx-\By\|}=\frac{\la\tdx-\tdy,\mathbf\nu_{\mathbf{x}}\ra}{\|\tdx-\tdy\|}, \quad \Bx,\By\in S_\delta^c.
$$
With the above facts and by using a similar argument to the proof of the first case of \eqnref{eq:arg1}, one can prove \eqnref{eq:arg1} for $\Bx\in S_\delta^c$, $c\in \{a ,b \}$.
To prove \eqnref{eq:arg2}, we first note that
\beq\label{eq:est111}
\nabla_{\mathbf{x}}\times\nabla_{\mathbf{x}}
\times \Acal_{D_\delta}^{\omega}[\ba](\Bx)=\omega^2 \Acal_{D_\delta}^{\omega}[\ba](\Bx)+\nabla_{\mathbf{x}}\nabla_{\mathbf{x}}\cdot\Acal_{D_\delta}^{\omega}[\ba](\Bx).
\eeq
By using (4.17) in \cite{DLU15}, one can easily obtain
\beq\label{eq:est112}
\Acal_{D_\delta}^\omega[\ba](\Bx)=\left\{
\begin{array}{ll}
\delta \Acal_{S^f}^\omega[\tilde\ba](\tdx)+ \Ocal(\delta^2\|\tilde\ba\|_{\mathrm{TH}^{-1/2}(\p D)}), &\Bx \in S_\delta^f, \\
\delta (\Acal_{S^c}+\Acal_{S^f}^\omega)[\tilde\ba](\tdx)+ \Ocal(\delta^2\|\tilde\ba\|_{\mathrm{TH}^{-1/2}(\p D)}), & \Bx \in S_\delta^c,
\end{array}
\right.
\eeq
for $ c\in\{a, b\}$, where $\Acal_{S^f}^\omega$ is defined similarly as $\Scal_{S^f}^\omega$ in \cite{DLU15}.
Using similar expansion strategy, one can prove for $\Bx\in S_\delta^f$
\beq\label{eq:est113}
\nabla_{\mathbf{x}}\nabla_{\mathbf{x}}\cdot\Acal_{D_\delta}^{\omega}[\ba](\Bx)=\delta\nabla_{\mathbf{z}_x}\nabla_{\mathbf{z}_x}\cdot\Acal_{S^f}^{\omega}[\ti\ba](\tdx)
+\Ocal(\delta^2\|\tilde\ba\|_{\mathrm{TH}^{-1/2}(\p D)}).
\eeq
Note that for $\Bx\in S_\delta^c$, $c\in\{a, b\}$, there holds (see Proposition 5.2 in \cite{ADM14} for details)
\beq\label{eq:est114}
\nabla_{\mathbf{x}}\nabla_{\mathbf{x}}\cdot\Acal_{D_\delta}^{\omega}[\ba](\Bx)
=\frac{1}{\delta}\nabla_{\ti{\mathbf{x}}}\nabla_{\ti{\mathbf{x}}}\cdot\Acal_{S^c}[\ti\ba](\tdx)+\Ocal(\delta\|\tilde\ba\|_{\mathrm{TH}^{-1/2}(\p D)}).
\eeq
By combining \eqnref{eq:est111}-\eqnref{eq:est114} together, one can complete the proof.
\end{proof}

\subsection{Asymptotic expansions}
In order to tackle the integral equation \eqref{eq:trans1}, we shall first derive some crucial asymptotic expansions. Henceforth, we denote $\tilde{\ba }(\tdy):=\ba (\By)$ for $\tdy=A(\By)$, $\By\in D_\delta$
and $\tdy\in \p D$. The same notation shall be adopted for $\Bx$ and $\tdx$.

For $\Bx\in \RR^3\setminus D_\delta$ with sufficiently large $\|\Bx\|$, we can expand $\bE_\delta-\bE^i$ from \eqnref{eq:vis1} in
 $\Bz\in \GG_0$ as follows
\begin{align}
&(\bE_\delta-\bE^i)(\Bx)=\nabla\times\Acal_{D_\delta}^{\omega}[\ba ](\Bx)=\nabla\times\int_{\p D_\delta}G_\omega(\Bx-\By)\ba (\By)\, d\Gs_y \nonumber\\
=& \nabla\times\int_{\p D_\delta}G_\omega(\Bx-\Bz_y)\ba (\By)\, d\Gs_y -\nabla\times\int_{\p D_\delta}\nabla G_\omega(\Bx-\Bz_y)\cdot(\By-\Bz_y)\ba (\By)\, d\Gs_y
\nonumber\\
&+ \nabla\times\int_{\p D_\delta}(\By-\Bz_y)^T\nabla^2G_\omega(\Bx-\mathbf\zeta(\By))(\By-\Bz_y)\ba (\By)\, d\Gs_{y}\nonumber\\
:=& R_1 +R_2 +R_3, \label{eq:I1I2}
\end{align}
where $\zeta(\By)=\eta \By+(1-\eta)\Bz_y \in D_\delta$ for some $\eta\in (0,1)$, and the superscript $T$ signifies the matrix transpose. We next estimate the three terms $R_1, R_2$ and $R_3$ in \eqref{eq:I1I2}.
The term $R_3$ in \eqref{eq:I1I2} is a remainder term from the Taylor series expansion and it verifies the following estimate,
\begin{align}
 &\|R_3\|_{L^{\infty}(\mathbb{S}^2)^3}
=\delta^2 \Big\|\nabla\times\int_{\p D_\delta}(\tdy-\Bz_{\tilde{y}})^T\nabla^2G_\omega(\Bx-\mathbf\zeta(\By))(\tdy-\Bz_{\tilde{y}})\ba (\By)\, d\Gs_{y}\Big\|_{L^{\infty}(\mathbb{S}^2)^3}
\nonumber\\
\leq &C\delta^3 \frac{1}{\|\Bx\|} \|\tilde\ba \|_{\mathrm{TH}^{-1/2}(\p D)}.\label{eq:R3}
\end{align}
In the sequel, we shall need the expansion of the incident plane wave $\bE^i$ in $\Bz\in \Gamma_0$, and there holds
\beq\label{eq:exp1}
\bE^i(\By)=\bE^i(\Bz_y)+\nabla \bE^i(\Bz_y)\cdot (\By-\Bz_y)+ \sum_{|\alpha|=2}^\infty \p_{\mathbf{y}}^\alpha \bE^i(\Bz_y)(\By-\Bz_y)^\alpha,
\eeq
where the multi-index $\alpha=(\alpha_1, \alpha_2, \alpha_3)$ and $\p_{\mathbf{y}}^{\alpha}=\p_{y_1}^{\alpha_1}\p_{y_2}^{\alpha_2}\p_{y_3}^{\alpha_3}$ with $\mathbf{y}=(y_1,y_2,y_3)$.
Since for $\By\in \p D_\delta$, $\mathbf\nu_{\mathbf{y}}=\mathbf\nu_{\tilde{\mathbf{y}}}$, one further has
$$
\mathbf\nu_{\mathbf{y}}\times \bE^i(\By)=\mathbf\nu_{\tilde{\mathbf{y}}}\times\sum_{|\alpha|=0}^\infty \delta^{\alpha} \p_{\mathbf{y}}^\alpha \bE^i(\Bz_y)(\tdy-\Bz_{\ti y})^\alpha.
$$
In what follows, we define
\begin{equation}\label{eq:ppp}
\tilde\Phi(\tdy):=\Phi(\By)=\mathbf{\nu}_{\mathbf{y}} \times  \bE_\delta(\By)\Big|_{\partial D_\delta}^+.
\end{equation}
\begin{thm}\label{prop:2}
Let $\bE_\delta$ be the solution to \eqnref{eq:sys1}, then there holds for $\Bx\in\mathbb{R}^3\backslash\overline{D}$,
\beq\label{eq:thmprin1}
\int_{S^f} G_\omega(\Bx-\Bz_{\tilde{y}}) \tilde\ba (\tdy)\, d\Gs_{\tilde{y}} = -2\int_{S^f} G_\omega(\Bx-\Bz_{\tilde{y}})\tilde\Phi(\tdy)\, d\Gs_{\tilde{y}} +\Ocal(\delta(\|\tilde\Phi\|_{\mathrm{TH}^{-1/2}(\p D)}+1)).
\eeq
If one assumes that $\Phi(\By)=0$, $\By\in \p D_\delta$, then there holds for $\Bx\in \RR^3\setminus \overline{D}$,
\begin{equation}\label{eq:aaa1}
\begin{split}
(\bE_\delta-\bE^i)(\Bx)=&2\delta^2\nabla\times\Big(\int_{S^f}G_{\omega}(\Bx-\Bz_{\ti y})\nu_{\tilde{\mathbf{y}}}\times \big(\nabla \bE^i(\Bz_{\tilde{y}})(\tdy-\Bz_{\ti{y}})\big)\, d\Gs_y\nonumber \\
& \quad \quad +\int_{S^a\cup S^b}G_{\omega}(\Bx-\Bz_{\tilde{y}})\nu_{\tilde{\mathbf{y}}}\times \bE^i(\Bz_{\tilde{y}}) \, d\Gs_{\tilde{y}}\nonumber\\
& \quad \quad  -\int_{S^f}\big(\nabla G_{\omega}(\Bx-\Bz_{\tilde{y}})\cdot(\tdy-\Bz_{\tilde{y}})\big)\nu_{\tilde{\mathbf{y}}}\times \bE^i(\Bz_{\tilde{y}})(\tdy)
\, d\Gs_{\tilde{y}}\Big)+\Ocal(\delta^3).
\end{split}
\end{equation}
\end{thm}

\begin{proof}
Recall that $-\frac{I}{2}+\Mcal_{D_\delta}^{\omega}$ is invertible on ${\mathrm{TH}}^{-1/2}_{\mathrm{div}} (\p D_\delta)$ (see, e.g. \cite{Grie08}).
By \eqnref{eq:trans1}, we see that
$$
\ba (\Bx)=\Big(-\frac{I}{2}+\Mcal_{D_\delta}^{\omega}\Big)^{-1}\Big[\mathbf\nu_{\mathbf{y}}\times (\bE_\delta-\bE^i)(\By)\Big|_{\partial D_\delta}^+\Big](\Bx).
$$
Using the results in Lemma \ref{le:1} and the expansion of $\bE^i$ in \eqref{eq:exp1}, we have for $\Bx\in S_\delta^f$ that
\begin{align}
\tilde\ba (\tdx)=&-2\ti\Phi(\tdx)+2 \nu_{\tilde{\mathbf{x}}}\times \bE^i(\Bz_{\tilde{x}}) - 4\delta
\Mcal_{S^f}^{\omega}[\ti\Phi](\tdx)+2\delta\nu_{\tilde{\mathbf{x}}}\times \big(\nabla \bE^i(\Bz_{\tilde{x}})(\tdx-\Bz_{\ti{x}})\big) \nonumber\\
&+ 4\delta\Mcal_{S^f}^{\omega}[\mathbf\nu_{\tilde{\mathbf{y}}}\times \bE^i(\Bz_{\tilde{y}})](\tdx)+\Ocal(\delta^2(\|\ti\Phi\|_{{\mathrm{TH}}^{-1/2}(\p D)}+1)).\label{eq:tmpexp1}
\end{align}
Note that
\beq\label{eq:techpt1}
\int_{S^f}G_\omega(\Bx-\Bz_{\tilde{y}})(\nu_{\tilde{\mathbf{y}}}\times \bE^i(\Bz_{\tilde{y}}))\, d\Gs_{\ti y}
=\int_{\GG_0}G_\omega(\Bx-\Bz_{\tilde{y}})\int_0^{2\pi}\nu_{\tilde{\mathbf{y}}}\, d\vartheta \times\bE^i(\Bz_{\tilde{y}})\, d r=0,
\eeq
where $(r,\vartheta)$ stand for the polar coordinates, and this
together with \eqnref{eq:tmpexp1} readily implies \eqnref{eq:thmprin1}.

Next, if $\tilde\Phi(\tdy)=\Phi(\By)=0$, then it can be seen from \eqnref{eq:arg1} and \eqref{eq:tmpexp1} that
$$
\|\tilde\ba \|_{\mathrm{TH}^{-3/2}(\p D)}\leq C,
$$
where $C$ is a positive constant depending only on $D$ and $\omega$.
By using our earlier results in \eqnref{eq:I1I2} and \eqref{eq:R3}, one can first show that
\begin{equation}\label{eq:ddd2}
\begin{split}
\Scal_{D_\delta}^{\omega}[\ba ](\Bx)
=& \int_{\p D_\delta}G_{\omega}(\Bx-\Bz_y)\ba (\By)\, d\Gs_y -\int_{\p D_\delta}\big(\nabla G_{\omega}(\Bx-\Bz_y)\cdot(\By-\Bz_y)\big)\ba (\By)\, d\Gs_y +\Ocal(\delta^3) \\
=&\delta\int_{S^f}G_{\omega}(\Bx-\Bz_{\tilde{y}})\tilde{\ba }(\tdy)d\Gs_{\tilde{y}} +\delta^2\Big(\int_{S^a\cup S^b}G_{\omega}(\Bx-\Bz_{\tilde{y}})\tilde{\ba }(\tdy)\, d\Gs_{\tilde{y}}\\
&-\int_{S^f}\big(\nabla G_{\omega}(\Bx-\Bz_{\tilde{y}})\cdot(\tdy-\Bz_{\tilde{y}})\big)\tilde{\ba }(\tdy)
\, d\Gs_{\tilde{y}}\Big)+\Ocal(\delta^3).
\end{split}
\end{equation}
By using \eqnref{eq:tmpexp1} we have for $\By\in S_\delta^f$ that
\beq\label{eq:tmpexp2}
\tilde\ba (\tdy)=2 \nu_{\tilde{\mathbf{y}}}\times \bE^i(\Bz_{\tilde{y}}) +2\delta\nu_{\tilde{\mathbf{y}}}\times \big(\nabla \bE^i(\Bz_{\tilde{y}})(\tdy-\Bz_{\ti{y}})\big)+\Ocal(\delta^2).
\eeq
Substituting \eqnref{eq:tmpexp2} into \eqnref{eq:ddd2} and using \eqnref{eq:techpt1} one can easily obtain
\begin{align*}
\Scal_{D_\delta}^{{\omega}}[\ba ](\Bx)
=&2\delta^2\Big(\int_{S^f}G_{\omega}(\Bx-\Bz_y)\nu_{\tilde{\mathbf{y}}}\times \big(\nabla \bE^i(\Bz_{\tilde{y}})(\tdy-\Bz_{\ti{y}})\big)\, d\Gs_y \\
& \quad \quad +\int_{S^a\cup S^b}G_{\omega}(\Bx-\Bz_{\tilde{y}})\nu_{\tilde{\mathbf{y}}}\times \bE^i(\Bz_{\tilde{y}}) \, d\Gs_{\tilde{y}}\\
& \quad \quad  -\int_{S^f}\big(\nabla G_{\omega}(\Bx-\Bz_{\tilde{y}})\cdot(\tdy-\Bz_{\tilde{y}})\big)\nu_{\tilde{\mathbf{y}}}\times \bE^i(\Bz_{\tilde{y}})(\tdy)
\, d\Gs_{\tilde{y}}\Big)+\Ocal(\delta^3),
\end{align*}
which then completes the proof by using \eqnref{eq:vis1}.
\end{proof}

We continue with the estimates of $R_1$ and $R_2$ in \eqnref{eq:I1I2} for the proposed full-cloaking structure with
arbitrary but regular $\varepsilon_a$, $\mu_a$ and $\Gs_a$ in $D_{\delta/2}$. In what follows, we let $C$ denote a generic positive constant. It may change from one inequality to another inequality in our estimates. Moreover, it may depend on different parameters, but it does not depend on $\varepsilon_a$, $\mu_a$, $\Gs_a$ and $\mathbf{p}$, $\mathbf{d}$, $\hat{\Bx}$.

We first note that, by using \eqref{eq:tmpexp1} and \eqnref{eq:arg1} in Lemma \ref{le:1}, there holds
\beq\label{eq:estphi11}
\|\tilde\ba \|_{\mathrm{TH}^{-1/2}(\p D)}\leq C (\|\tilde\Phi\|_{\mathrm{TH}^{-1/2}(\p D)} +1).
\eeq
Next, by taking expansion around $\Bz\in \GG_0$ and using \eqnref{eq:thmprin1} one can show that for $\delta\in\mathbb{R}_+$ sufficiently small and $\|\Bx\|$ sufficiently large,
\begin{equation}\label{eq:ccc1}
\begin{split}
\|R_1\|=&\Big\|\nabla\times\int_{\p D_\delta}G_\omega(\Bx-\Bz_y)\ba (\By)\, d\Gs_y\Big\|\\
\leq& \delta\Big\|\nabla\times\int_{S^f} G_\omega(\Bx-\Bz_{\tilde{y}})
\tilde{\ba }(\tdy)\, d\Gs_{\tilde{y}}\Big\|+ C\delta^2\frac{1}{\|\Bx\|}\|\tilde\ba \|_{\mathrm{TH}^{-1/2}(\p D)}\\
\leq&C\delta\Big\|\nabla\times\int_{S^f} G_\omega(\Bx-\Bz_{\tilde{y}})
\tilde{\Phi}(\tdy)\, d\Gs_{\tilde{y}}\Big\|+ C\delta^2\frac{1}{\|\Bx\|}(\|\tilde\Phi\|_{\mathrm{TH}^{-1/2}(\p D)}+1) \\
\leq & C\delta\frac{1}{\|\Bx\|} \|\tilde{\Phi}\|_{\mathrm{TH}^{-1/2}(S^f)}+ C\delta^2\frac{1}{\|\Bx\|}(\|\tilde\Phi\|_{\mathrm{TH}^{-1/2}(\p D)}+1).
\end{split}
\end{equation}
By using Taylor's expansions again and \eqnref{eq:tmpexp1}, one can show the following estimation for $R_2$,
\begin{align}
\|R_2\|=& \Big\|\nabla\times\int_{\p D_\delta}\nabla G_\omega(\Bx-\Bz_y)\cdot(\By-\Bz_y)\ba (\By)\, d\Gs_y\Big\|\nonumber\\
\leq & C\delta^2\frac{1}{\|\Bx\|}(\|\tilde\Phi\|_{\mathrm{TH}^{-1/2}(S^f)}+1)+C\delta^3\frac{1}{\|\Bx\|} (\|\tilde\Phi\|_{\mathrm{TH}^{-1/2}(\p D)}+1). \label{eq:sharp1}
\end{align}
Hence, by applying the estimates in \eqref{eq:R3}, \eqref{eq:ccc1} and \eqref{eq:sharp1} to \eqref{eq:I1I2}, we have
\beq\label{eq:mainest1}
\|\bE_\delta-\bE^i\|\leq C\frac{\delta}{\|\Bx\|} \|\tilde{\Phi}\|_{\mathrm{TH}^{-1/2}(S^f)}+C\delta^2\frac{1}{\|\Bx\|}(\|\tilde\Phi\|_{\mathrm{TH}^{-1/2}(\p D)}+1)
\eeq
for $\|\Bx\|$ sufficiently large.
With \eqnref{eq:mainest1} at hand, one readily has
\begin{lem}\label{le:farfield51}
The scattering amplitude $\mathbf{A}^{\delta}_\infty(\hat{\Bx}; \mathbf{p}, \mathbf{d})$ corresponding to the scattering configuration described in Theorem~\ref{th:main1} satisfies
\beq\label{eq:estfnl1}
\|\mathbf{A}^{\delta}_\infty(\hat{\Bx}; \mathbf{p},\mathbf{d})\|\leq C\delta \|\tilde{\Phi}\|_{\mathrm{TH}^{-1/2}(S^f)}+C\delta^2(\|\tilde\Phi\|_{\mathrm{TH}^{-1/2}(\p D)}+1),
\eeq
where $C$ depends only on $D$ and $\omega$.
\end{lem}

\subsection{Proof of Theorem~\ref{th:main1}}
By Lemma \ref{le:farfield51}, it is straightforward to see that in order to derive the estimate of the scattering
amplitude $\mathbf{A}_\infty^\delta$, it suffices for us to derive the corresponding estimate of $\|\ti\Phi\|_{\mathrm{TH}^{-1/2}}$. To that end, we shall need the following lemma, whose proof can be found in \cite{BLZ}.
\begin{lem}\label{le:estL1}
Let $B_R$ be a central ball of radius $R$ such that $D_\delta\Subset B_R$. Then the solutions
$E_\delta, H_\delta\in H_{loc}(\mbox{\emph{curl}};\RR^3)$ to \eqnref{eq:sys1} verify
\begin{align}
&\int_{D_\delta\setminus\overline{D}_{\delta/2}} \Gs_l \bE_\delta\cdot \overline{\bE}_\delta\, d\Bx +
\int_{D_{\delta/2}} \Gs_a \bE_\delta\cdot \overline{\bE_\delta}\, d\Bx\nonumber \\
=&\int_{\p B_R}(\nu\times\overline{\bE_\delta^+})\cdot (\nu\times \nu\times \bH_\delta^+)\, d\Gs_x+
\Re\int_{\p D_\delta}(\nu\times\overline{\bE^i})\cdot(\nu\times \nu\times \bH_\delta^+)\Big|_+\, d\Gs_x\nonumber\\
&+ \Re\int_{\p D_\delta} (\nu\times\overline{\bE_\delta^+})\Big|_+\cdot(\nu\times\nu\times \bH^i)\, d\Gs_x. \label{eq:estL1}
\end{align}
\end{lem}
\begin{rem}
It is remarked that the last two terms in the RHS of \eqref{eq:estL1} in \cite{BLZ} were
$$\Re\int_{\p B_R}(\nu\times\overline{\bE^i})\cdot(\nu\times \nu\times \bH_\delta^+) d\Gs_x
+ \Re\int_{\p B_R} (\nu\times\overline{\bE_\delta^+})\cdot(\nu\times\nu\times \bH^i)\, d\Gs_x .$$
Here, we modify the two terms for the convenience of the present study.
\end{rem}

Define a boundary operator $\GL$ such that
\beq\label{eq:bndop1}
\GL(\nu\times \bE_\delta|_{\p B_R})=\nu\times \bH_\delta|_{\p B_R} : \quad \mathrm{TH}_{\mathrm{div}}^{-1/2}(\p B_R)\rightarrow \mathrm{TH}_{\mathrm{div}}^{-1/2}(\p B_R).
\eeq
It is well-known that $\GL$ is a bounded operator in $\mathrm{TH}_{\mathrm{div}}^{-1/2}(\p B_R)$ (cf. \cite{CK,Ned}).
By using Lemma \ref{le:estL1}, one can show that
\begin{lem}
Let $\bE_\delta$ and $\bH_\delta$ be solutions to the system \eqnref{eq:sys1}, where $D_\delta$ is the virtual domain
described at the beginning of this section, then we have
\begin{align}
\int_{D^f_\delta\setminus \overline{D}_{\delta/2}} \|\bE_\delta\|^2\, d\Bx \leq & C \delta^{1-t}
\Big(\|\nu\times \bE_\delta^+\|_{\mathrm{TH}_{\mathrm{div}}^{-1/2}(\p B_R)}^2
+\delta\|\ti\ba\|_{\mathrm{TH}^{-1/2}(\p D)}\Big), \label{estL2}
\end{align}
and
\begin{align}
\int_{D^c_\delta\setminus \overline{D}_{\delta/2}} \|\bE_\delta \|^2\, d\Bx \leq & C \delta^{2-t}\Big(\|\nu\times \bE_\delta^+\|_{\mathrm{TH}_{\mathrm{div}}^{-1/2}(\p B_R)}^2
+\delta\|\ti\ba\|_{\mathrm{TH}^{-1/2}(\p D)}\Big), \ c\in\{a, b\},\label{estL3}
\end{align}
where the constant $C$ depends only on $D$, $\omega$ and $R$.
\end{lem}
\begin{proof}
First, by \eqnref{eq:estL1} and the fact that $\GL$ is bounded in $\mathrm{TH}_{\mathrm{div}}^{-1/2}(\p B_R)$, there holds
\begin{align*}
\int_{D_\delta\setminus\overline{D}_{\delta/2}} \Gs_l \bE_\delta\cdot \overline{\bE}_\delta\, d\Bx \leq
C\|\nu\times \bE_\delta^+\|_{\mathrm{TH}_{\mathrm{div}}^{-1/2}(\p B_R)}^2
+\mathbb{R}_1 + \mathbb{R}_2,
\end{align*}
where by using \eqnref{eq:vis2} and \eqnref{eq:arg2} one further has
\begin{align*}
\mathbb{R}_1&=\Big|\Re\int_{\p D_\delta}(\nu\times\overline{\bE^i})\cdot(\nu\times \nu\times \bH_\delta^+)\Big|_+ d\Gs_x\Big|
=\Big|\Re\int_{\p D_\delta}\overline{\bE^i}\cdot(\nu\times \bH_\delta^+)\Big|_+\, d\Gs_x\Big|\\
&\leq \frac{\delta}{\omega}\Big |\int_{S^a\cup S^b}\overline{\bE^i}(\Bz_{\ti{\mathbf{x}}})\cdot(\nu_{\ti{\mathbf{x}}}\times\nabla_{\ti{\mathbf{x}}}\times\nabla_{\ti{\mathbf{x}}}\times\Acal_{S^{a}\cup S^{b}}[\ti\ba](\tdx))\, d\Gs_{\ti x}\Big| +C\delta^2\|\ti\ba\|_{\mathrm{TH}^{-1/2}(\p D)},
\end{align*}
and by using \eqnref{eq:arg1} one further has
\begin{align*}
\mathbb{R}_2&=\Big|\Re\int_{\p D_\delta} (\nu\times\overline{\bE_\delta^+})\Big|_+\cdot(\nu\times\nu\times \bH^i)\, d\Gs_x\Big|
=\Big|\Re\int_{\p D_\delta}\overline{\bE_\delta^+}\Big|_+\cdot(\nu\times \bH^i)\, d\Gs_x\Big|\\
&\leq C\delta^2\|\ti\ba\|_{\mathrm{TH}^{-1/2}(\p D)}.
\end{align*}
Combining those with \eqnref{eq:jocob11}, \eqnref{eq:eigenB1}, together with the use of the definition of $\Gs_l$ in \eqnref{eq:loss1}, we can complete the proof.
\end{proof}

It is remarked that in the estimate \eqref{estL3}, the dependence on the artificial $R$ of the generic constant $C$ can obviously be absorbed into the dependence on $D$.

In what follows, we shall estimate $\|\tilde\Phi\|_{\mathrm{TH}^{-1/2}(S^f)}$ and $\|\tilde\Phi\|_{\mathrm{TH}^{-1/2}(S^a\cup S^b)}$, separately.
Clearly, it suffices to estimate  $\|\tilde\Phi\|_{H^{-1/2}(S^f)^3}$ and $\|\tilde\Phi\|_{H^{-1/2}(S^a\cup S^b)^3}$, respectively.
We recall that the $H^{-1/2}(\p D)^3$-norm of the function $\tilde\Phi(\tdx)$ is defined as follows
\beq\label{eq:defnorm1}
\|\tilde\Phi\|_{H^{-1/2}(\p D)^3}= \sup_{\|\varphi\|_{H^{1/2}(\p D)^3\leq 1}}\Big|\int_{\p D} \tilde\Phi(\tdx)\cdot\overline{\varphi}(\tdx)\, d\Gs_{\tilde{x}}\Big|.
\eeq
Moreover, the $H^{-1/2}(S^c)^3$-norm of $\tilde\Phi$ for $c\in\{f, a, b\}$ is given as
\beq\label{eq:defnorm2}
\|\tilde\Phi\|_{H^{-1/2}(S^c)^3}:=\sup_{\|\varphi\|_{H^{1/2}_0(S^c)^3}\leq 1} \Big|\int_{S^c} \tilde\Phi(\tdx)\cdot\overline{\varphi}(\tdx)\, d\Gs_{\tilde{x}}\Big|,
\eeq
where $H^{1/2}_0(S^c)^3$ denotes the set of $H^{1/2}(S^c)^3$-functions which have zero extensions to the whole boundary $\partial D$.
We refer to \cite{Ada,Lio,Wlok87} for more relevant discussions on the Sobolev spaces.
To proceed, we shall first establish the following important auxiliary Sobolev extension result, which is a localized version of Lemma 3.4 in \cite{BLZ}.

\begin{lem}\label{le:extenlocal1}
Suppose $D$ is a simply connected domain with a $C^3$-smooth boundary $\p D$ and $S^c\Subset\p D$ is an open surface lying on $\p D$.
Let $D'$ be a simply connected
domain with a $C^3$-smooth boundary $\p D'$ which satisfies $S^c\cap \p D'=\emptyset$ and $D'\subset D$. Then for any $\psi\in H_0^{1/2}(S^c)^3$, there exists $\BW\in H^2(D)^3$ such that
\begin{align*}
&\nu\times \BW =0 & \mbox{on}\quad S^c, \\
&\nu\times(\nu \times (\nabla\times \BW)) = \nu\times (\nu\times \psi) & \mbox{on}\quad S^c, \\
&\|\BW\|_{H^2(D)^3}\leq C \|\psi\|_{H^{1/2}(S^c)^3}, &\\
&\BW=0 & \mbox{in} \quad D',
\end{align*}
where $C$ is a constant depending only on $D$.
\end{lem}
\begin{proof}
By zero-extending $\psi$ to $\p D$, we first have by Lemma 3.4 in \cite{BLZ} that
there exists $\BW_0\in H^2(D)^3$ satisifying
\begin{align*}
&\nu\times \BW_0 =0 & \mbox{on}\quad S^c, \\
&\nu\times(\nu \times (\nabla\times \BW_0)) = \nu\times (\nu\times \psi) & \mbox{on}\quad S^c, \\
&\|\BW_0\|_{H^2(D)^3}\leq C \|\psi\|_{H^{1/2}(S^c)^3}, &
\end{align*}
where  $C$ is a constant depending only on $D$.
Let $\mathfrak{S}=supp(\psi)$ denote the support of the function $\psi$. By definition there holds $\mathfrak{S}\Subset S^c$.
Next, we use exactly the same strategy as that in Addendum of Theorem 14.1 in \cite{Wlok87}. We cover $\mathfrak{S}$ with
finitely many, say $\mathfrak{N}$, small enough regions $\{\mathfrak{S}_j\}_{j=1}^{\mathfrak{N}}$ such that
$(\bigcup_{j=1}^\mathfrak{N}\mathfrak{S}_j)\cap \mathfrak{S}\subset S^c$. Then in each $\mathfrak{S}_j\cap D$ there
exits an appropriately chosen $\BW_j\in H^2((\bigcup_{j=1}^\mathfrak{N}\mathfrak{S}_j)\cap D)^3$
(see formula (9) in Addendum in \cite{Wlok87}) such that
\begin{align*}
&\nu\times \BW_j =0 & \mbox{on}\quad \p (\mathfrak{S}_j\cap D)\cap S^c, \\
&\nu\times(\nu \times (\nabla\times \BW_j)) = \nu\times (\nu\times \psi) & \mbox{on}\quad \p (\mathfrak{S}_j\cap D)\cap S^c, \\
&\|\BW_j\|_{H^2((\bigcup_{j=1}^\mathfrak{N}\mathfrak{S}_j)\cap D)^3}\leq C \|\psi\|_{H^{1/2}(S^c)^3}, &
\end{align*}
where $C$ is a constant depending only on $D$.
Finally, by setting $\BW:=\sum_{j=1}^\mathfrak{N} \BW_j$ in $(\bigcup_{j=1}^\mathfrak{N}\mathfrak{S}_j)\cap D$ and
$\BW=0$ in $D\setminus ((\bigcup_{j=1}^\mathfrak{N}\mathfrak{S}_j)\cap D)$, one can complete the proof.
\end{proof}

We proceed with the proof of Theorem~\ref{th:main1}.

\begin{lem}\label{le:estiphi41}
Let $\tilde\Phi$ be defined in \eqnref{eq:ppp}, where $(\bE_\delta, \bH_\delta)$ are the solutions
to \eqnref{eq:sys1} with the corresponding $\varepsilon_\delta$, $\mu_\delta$ and $\Gs_\delta$ given by \eqnref{eq:loss1}. Then there hold
\beq\label{eq:estphi1}
\| \tilde\Phi\|_{\mathrm{TH}^{-1/2}(S^f)}\leq
C\delta^{-1+\beta}\|\bE_\delta\|_{L^2(D_\delta^f\setminus D_{\delta/2})^3},
\eeq
and
\beq\label{eq:estphi2}
\|\tilde\Phi\|_{\mathrm{TH}^{-1/2}(S^c)}
\leq C\delta^{-3/2+\beta'}\|\bE_\delta\|_{L^2(D_\delta^c\setminus D_{\delta/2})^3}
, \quad c\in\{a, b\},
\eeq
where $\beta=\min\{1, -1+r+s, -1+t+s\}$ and $\beta'=\min\{1, -2+r+s, -2+t+s\}$.
\end{lem}

\begin{proof}
It suffices to show that the same estimates in \eqnref{eq:estphi1} and
\eqnref{eq:estphi2} hold for $\|\tilde\Phi(\tdx)\|_{H^{-1/2}(S^c)^3}$, $c\in\{f, a, b\}$.
For any test function
$\psi\in H^{1/2}_0(S^c)^3$, $c\in \{f, a, b\}$, we introduce an auxiliary function $\BW\in H^2(D)^3$ which satisfies the conditions in Lemma \ref{le:extenlocal1}
($D'$ is chosen to be $D'=D_{1/2}\cup (D\setminus D^c)$).
First, it follows from the properties of $\BW$ that
\begin{align*}
&\int_{S^f} \tilde\Phi(\tdx)\cdot\psi(\tdx)\, d\Gs_{\tilde{x}}=
 - \int_{S^f} (\nu\times\ti\bE_\delta) \cdot (\nu\times(\nu\times\psi))\, d\Gs_{\tilde{x}} \\
=& -\int_{S^f} \ti\bE_\delta \cdot (\nu\times(\nabla\times \BW))d\Gs_{\tilde{x}}- \int_{\p D^f\setminus S^f} \ti\bE_\delta \cdot (\nu\times(\nabla\times \BW))\, d\Gs_{\tilde{x}} \\
=& -\int_{\p D^f} \ti\bE_\delta \cdot (\nu\times(\nabla\times \BW))\, d\Gs_{\tilde{x}}.
\end{align*}
From its construction, it is readily seen that $\nu\times \BW|_{\p D^c}=0$. Define
\beq\label{eq:curlw1}
\widehat{\mbox{Curl}\, W}:=(\ti \BB^T)^{-1}(\nabla\times \BW).
\eeq
By using \eqnref{eq:ppp41}, \eqnref{eq:imp411}, \eqnref{eq:jacobdet1} and integration by parts, one also has
\begin{align*}
&\int_{S^f} \tilde\Phi(\tdx)\cdot\psi(\tdx)\, d\Gs_{\tilde{x}}=  -\delta^{-1}\int_{\p D^f} \hat\bE_\delta \cdot \Big(\nu\times \widehat{\mbox{Curl}\, W}\Big)\, d\Gs_{\tilde{x}} \\
= &-\delta^{-1}\left(\int_{\p D^f} \hat\bE_\delta \cdot \Big(\nu\times\widehat{\mbox{Curl}\, W}\Big)\, d\Gs_{\tilde{x}}
-\int_{\p D^f} \BW \cdot \Big(\nu\times\widehat{\mbox{Curl}\, W}\Big)d\Gs_{\tilde{x}}\right)\\
= & \delta^{-1}\left(\int_{D^f} \BW \cdot \Big(\nabla\times\widehat{\mbox{Curl}\, W}\Big)d\tilde{\Bx}
-\int_{D^f}\hat\bE_\delta \cdot \Big(\nabla\times\widehat{\mbox{Curl}\, W}\big)\Big)d\ti{\Bx}\right)\\
=& \delta^{-1}\int_{D^f\setminus \overline{D}_{1/2}}\BW \cdot \Big(\nabla\times\widehat{\mbox{Curl}\, W}\Big)d\tilde{\Bx}
-\delta\int_{D^f\setminus \overline{D}_{1/2}} \ti\bE_\delta \cdot \Big(\nabla\times(\nabla\times \BW)\Big)d\ti{\Bx}.
\end{align*}
For $\Bx\in D_{\delta}\setminus D_{\delta/2}$ there holds
\begin{align*}
 \nabla_{\mathbf{x}} \times \bE_\delta = i \omega \mu_l \bH_\delta , \quad
  \nabla_{\mathbf{x}}  \times \bH_\delta= - i \omega \Big(\varepsilon_l+i\frac{\Gs_l}{\omega}\Big) \bE_\delta.
\end{align*}
By using change of variables and \eqnref{eq:ppp41} in Lemma \ref{le:changerela1}, one can derive that
\begin{align*}
 |\ti\BB|\ti\BB^{-1} \nabla_{\ti{\mathbf{x}}} \times \hat\bE_\delta = i \omega \ti\mu_l \ti\bH_\delta , \quad
 |\ti\BB|\ti\BB^{-1} \nabla_{\ti{\mathbf{x}}}  \times \hat\bH_\delta= - i \omega \Big(\ti\varepsilon_l+i\frac{\ti\Gs_l}{\omega}\Big) \ti\bE_\delta,
\end{align*}
which holds in $D\setminus \overline{D}_{1/2}$. By combining \eqnref{eq:loss1}, \eqnref{eq:def3} and \eqnref{eq:curlw1}, one can further show that in $D\setminus \overline{D}_{1/2}$,
\beq\label{eq:gtform1}
\nabla\times\widehat{\mbox{Curl}\, W}=\omega^2(\delta^{r+s}+i\delta^{t+s}\frac{1}{\omega})\ti\bE_\delta.
\eeq
Thus by using \eqnref{eq:gtform1} one has
\begin{align*}
&\int_{S^f} \tilde\Phi(\tdx)\cdot\psi(\tdx)\, d\Gs_{\tilde{x}}\\
= &
\omega^2(\delta^{-1+r+s}+i\delta^{-1+t+s}\frac{1}{\omega})\int_{D^f\setminus \overline{D}_{1/2}} \BW \cdot \ti\bE_\delta\, d\tilde{\Bx}
-\delta\int_{D^f\setminus \overline{D}_{1/2}} \ti\bE_\delta \cdot \Big(\nabla\times(\nabla\times \BW)\Big)\, d\ti{\Bx},
\end{align*}
which in combination with the fact that
\begin{align*}
\|\ti\bE_\delta\|_{L^2(D^f\setminus \overline{D}_{1/2})^3}=\delta^{-1}\|\bE_\delta\|_{L^2(D_\delta^f\setminus \overline{D}_{\delta/2})^3},
\end{align*}
immediately yields that
\begin{align*}
\Big|\int_{S^c} \tilde\Phi(\tdx)\cdot\psi(\tdx)\, d\Gs_{\tilde{x}}\Big|\leq & C\delta^{\beta}\|\ti\bE_\delta\|_{L^2(D^f\setminus D_{1/2})^3}
\|\BW\|_{H^2(D^f\setminus D_{1/2})^3}\\
\leq & C\delta^{-1+\beta}\|\bE_\delta\|_{L^2(D_\delta^f\setminus D_{\delta/2})^3}
\|\BW\|_{H^2(D^f\setminus D_{1/2})^3},
\end{align*}
where $\beta=\min\{1, -1+r+s, -1+t+s\}$.
By using exactly the same strategy as above with $S^f$ and \eqnref{eq:imp412}, one can show that for $c\in\{a, b\}$
\begin{align*}
&\int_{S^c} \tilde\Phi(\tdx)\cdot\psi(\tdx)\, d\Gs_{\tilde{x}}\\
=&
\delta^{-2}\int_{D^c\setminus \overline{D}_{1/2}}\BW \cdot \Big(\nabla\times\widehat{\mbox{Curl}\, W}\Big)\, d\tilde{\Bx}
-\delta\int_{D^c\setminus \overline{D}_{1/2}} \ti\bE_\delta \cdot \Big(\nabla\times(\nabla\times \BW)\Big)d\ti{\Bx},
\end{align*}
which together with the fact
\begin{align*}
\|\ti\bE_\delta\|_{L^2(D^c\setminus \overline{D}_{1/2})^3}=\delta^{-3/2}\|\bE_\delta\|_{L^2(D_\delta^c\setminus \overline{D}_{\delta/2})^3}, \quad c\in\{a, b\},
\end{align*}
readily yields that
\begin{align*}
\Big|\int_{S^c} \tilde\Phi(\tdx)\cdot\psi(\tdx)\, d\Gs_{\tilde{x}}\Big|
\leq & C\delta^{-3/2+\beta'}\|\bE_\delta\|_{L^2(D_\delta^c\setminus \overline{D}_{\delta/2})^3}
\|\BW\|_{H^2(D^c\setminus \overline{D}_{1/2})^3},
\end{align*}
where $\beta'=\min\{1, -2+r+s, -2+t+s\}$. The proof can be completed by noting the definitions \eqnref{eq:defnorm1}, \eqnref{eq:defnorm2} and the fact that
$$
\|\BW\|_{H^2(D^c\setminus \overline{D}_{1/2})^3}\leq C \|\psi\|_{H^{1/2}(\p D^c)^3}, \quad c\in\{f, a, b\}.
$$
\end{proof}

\begin{proof}[Proof of Theorem \ref{th:main1}]
It is straightforward to see from \eqnref{eq:estfnl1} that we only need to derive the estimates of
$\|\tilde{\Phi}\|_{\mathrm{TH}^{-1/2}(S^f)}$ and $\|\tilde{\Phi}\|_{\mathrm{TH}^{-1/2}(\p D)}$.
To that end, we have by using \eqnref{eq:vis1}, \eqnref{eq:def1},
\eqnref{eq:tmpexp2} and \eqnref{eq:estphi11} that
\begin{equation}\label{eq:uunn1}
\begin{split}
\|\nu\times \bE_\delta^+\|_{\mathrm{TH}_{\mathrm{div}}^{-1/2}(\p B_R)}&\leq C \|\nu\times
\nabla\times \Acal_{D_\delta}[\ba]\|_{\mathrm{TH}_{\mathrm{div}}^{-1/2}(\p B_R)}\\
&\leq C\delta \|\ti\ba\|_{\mathrm{TH}^{-1/2}(S^f)}+C\delta^2\|\ti\ba\|_{\mathrm{TH}^{-1/2}(\p D)} \\
& \leq C\delta \|\tilde{\Phi}\|_{\mathrm{TH}^{-1/2}(S^f)}+C\delta^2(\|\tilde\Phi\|_{\mathrm{TH}^{-1/2}(\p D)}+1).
\end{split}
\end{equation}
Next, by combining \eqref{eq:uunn1}, \eqnref{eq:estphi11}, \eqnref{estL2}, \eqnref{estL3}, \eqnref{eq:estphi1} and \eqnref{eq:estphi2}, together with the use of the facts that both $\GL$ and $\GL^-$ are bounded, and $\beta'\leq\beta$ (see Lemma \ref{le:estiphi41}), one can further deduce that
\begin{align*}
\|\tilde\Phi\|_{\mathrm{TH}^{-1/2}(\p D)} \leq&
C\delta^{-1+\beta}\|\bE_\delta\|_{L^2(D_\delta^f\setminus D_{\delta/2})^3}
+ C\delta^{-3/2+\beta'}\|\bE_\delta\|_{L^2(D_\delta^{a\cup b}\setminus D_{\delta/2})^3}\\
\leq& C \delta^{-1+\beta+1/2-t/2}\Big(\|\nu\times \bE_\delta^+\|_{\mathrm{TH}_{\mathrm{div}}^{-1/2}(\p B_R)}^2
 + \Ocal(\delta\|\ti\ba\|_{\mathrm{TH}^{-1/2}(\p D)})\Big)^{1/2} \\
& + C\delta^{-3/2+\beta'+ 1-t/2}\Big(\|\nu\times \bE_\delta^+\|_{\mathrm{TH}_{\mathrm{div}}^{-1/2}(\p B_R)}^2
 + \Ocal(\delta\|\ti\ba\|_{\mathrm{TH}^{-1/2}(\p D)})\Big)^{1/2} \\
\leq& C\delta^{\beta'-t/2-1/2}(\delta \|\tilde{\Phi}\|_{\mathrm{TH}^{-1/2}(S^f)}+\delta^2\|\tilde\Phi\|_{\mathrm{TH}^{-1/2}(\p D)}+\delta^2)\\
&+C\delta^{\beta'-t/2}(\|\tilde\Phi\|_{\mathrm{TH}^{-1/2}(\p D)}^{1/2}+1),
\end{align*}
which in turn yields (noting that $\beta'-t/2\geq 0$)
\begin{align*}
\|\tilde\Phi\|_{\mathrm{TH}^{-1/2}(\p D)}& \leq C\delta^{\beta'-t/2}(\|\tilde\Phi\|_{\mathrm{TH}^{-1/2}(\p D)}^{1/2}+1).
\end{align*}
Thus one has
\beq\label{eq:impfnlest1}
\|\tilde\Phi\|_{\mathrm{TH}^{-1/2}(\p D)} \leq C\delta^{\beta'-t/2}.
\eeq
It is remarked that in \eqref{eq:impfnlest1}, the generic constant obviously depends on $R$, but as was remarked earlier that such a dependence can be absorbed into the dependence on $D$. Next, by using \eqnref{eq:estphi1} and \eqnref{eq:impfnlest1}, there holds
\begin{align*}
\|\tilde\Phi\|_{\mathrm{TH}^{-1/2}(S^f)} & \leq
C\delta^{-1+\beta}\|\bE_\delta\|_{L^2(D_\delta^f\setminus D_{\delta/2})^3}\\
&\leq C \delta^{-1+\beta+1/2-t/2}\Big(\|\nu\times \bE_\delta^+\|_{\mathrm{TH}_{\mathrm{div}}^{-1/2}(\p B_R)}^2
 + \Ocal(\delta\|\ti\ba\|_{\mathrm{TH}^{-1/2}(\p D)})\Big)^{1/2} \\
&\leq C\delta^{-1/2+\beta-t/2}(\delta \|\tilde{\Phi}\|_{\mathrm{TH}^{-1/2}(S^f)}+\delta^2)+C\delta^{\beta-t/2},
\end{align*}
and thus
\beq\label{eq:impfnlest2}
\|\tilde\Phi\|_{\mathrm{TH}^{-1/2}(S^f)} \leq C\delta^{\beta-t/2}.
\eeq
By plugging \eqnref{eq:impfnlest1} and \eqnref{eq:impfnlest2} into \eqnref{eq:estfnl1}, one finally has \eqnref{eq:ncest1}.

The proof is complete.
\end{proof}

\section{Regularized partial-cloaking of EM waves}\label{sect:5}

In this section, we consider the regularized partial-cloaking of EM waves by taking the generating set $\Gamma_0$ to be a flat subset on a plane in $\mathbb{R}^3$. In order to ease our exposition, we stick our subsequent study to a specific example considered in \cite{DLU15,LiLiuRonUhl} for the partial-cloaking of acoustic waves, where $\Gamma_0$ is taken to be a square on $\mathbb{P}_2:=\{\mathbf{x}=(x_1,x_2,x_3)\in\mathbb{R}^3: x_3=0\}$. The virtual domain is constructed as follows; see Fig.~\ref{fig2} for a schematic illustration.
Let $\mathbf{n}\in\mathbb{S}^2$ be the unit normal vector to $\GG_0$. Let
\begin{equation}\label{eq:rrr1}
D_q^0:=\GG_0(\Bx)\times [\Bx-\tau\cdot \mathbf{n}, \Bx+\tau\cdot \mathbf{n}],\quad \Bx\in\overline{\GG}_0,\quad 0\leq \tau\leq q,
\end{equation}
where we identify $\GG_0$ with its parametric representation $\GG_0(\Bx)$. We denote by $D_q^{1}$ the union of the four side
half-cylinders and $D_q^{2}$ the union of the four corner quarter-balls in Fig.~\ref{fig2}. Let
\begin{equation}\label{eq:vtq}
D_q:=D_q^0\cup D_q^1\cup D_q^2.
\end{equation}
Set $S_q^1$ and $S_q^2$ to denote
$$
S_q^1:=\p D_q \cap \p D_q^1, \quad S_q^2:=\p D_q \cap \p D_q^2,
$$
and
$$D_q^f:=D_q^1 \cup  D_q^2.$$
The upper and lower-surfaces of $D_q$ are respectively denoted by
$$
\GG_q^1:=\{\Bx+ q\cdot \mathbf{n};\, \Bx\in \GG_0\} \quad\mbox{and}\quad \GG_q^2:=\{\Bx-q\cdot \mathbf{n};\, \Bx\in \GG_0\}.
$$
Define $S_q^0:=\GG_q^1\cup \GG_q^2$.
We then have $\partial D_q=S_q^0\cup S_q^1\cup S_q^2$. Similar to our notations in Section~\ref{sect:4}, we let $\delta\in\mathbb{R}_+$ denote the asymptotically small regularization parameter and $D_\delta$ be the virtual domain. Clearly, we have
\beq\label{eq:struc51}
\p D_\delta= S_\delta^0 \cup S_\delta^1 \cup S_\delta^2.
\eeq
In what follows, if $q\equiv 1$, we drop the dependence on $q$ of $D_q, S_q^0$,
$S_q^1$ and $S_q^2$, and simply write them as $D, S^0$, $S^1$, and $S^2$. In concluding the description of the virtual domain for the partial-cloaking construction, we would like to emphasize that our subsequent study can be easily extended to a much more general case where the generating set $\Gamma_0$ can be a bounded subset on a plane with a convex and piecewise smooth boundary (in the topology induced from the plane).
\begin{figure}[h]
\begin{center}
  \includegraphics[width=3.5in,height=2.0in]{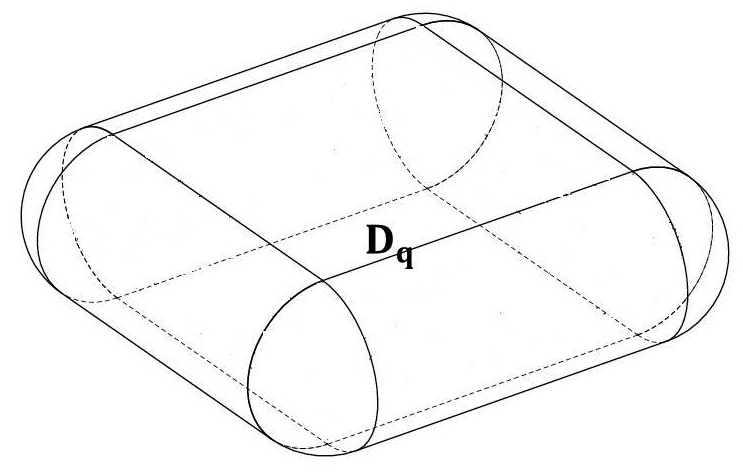}\\
  \includegraphics[width=3.8in,height=2.3in]{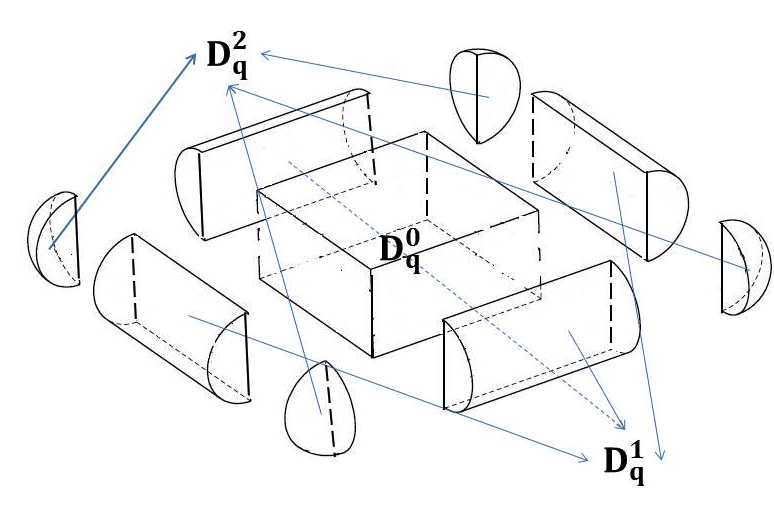}
  \end{center}
  \caption{Schematic illustration of the domain $D_q$ for the regularized partial-cloak. \label{fig2}}
\end{figure}

We introduce a blowup transformation $A$ which maps $D_\delta$ to $D$ exactly as that in \cite{LiLiuRonUhl}.
We stress that in $D_\delta^0$ the blow-up transformation takes the following form
$$
A(\By)=\tdy:=\left(\frac{\mathbf{e}_3\mathbf{e}_3^T}{\delta}+
\mathbf{e}_1\mathbf{e}_1^T+\mathbf{e}_2\mathbf{e}_2^T\right)\By, \quad \By\in D_\delta^0.
$$
where $\By\in D_\delta$ and $\tdy\in D$, and the three Euclidean unit vectors are as follows
$$
\mathbf{e}_1=(1, 0, 0)^T,\quad \mathbf{e}_2=(0, 1, 0)^T, \quad \mathbf{e}_3=(0, 0, 1)^T.
$$
We are now in the position of introducing the lossy layer for our partial-cloaking device
\beq\label{eq:loss2}
\begin{array}{ll}
\varepsilon_\delta(\Bx)=\varepsilon_l(\Bx):=\delta^{r}|\BB|\BB^{-1}, &
\mu_\delta(\Bx)=\mu_l(\Bx):=\delta^{s}|\BB|\BB^{-1}, \\
\Gs_\delta(\Bx)=\Gs_l(\Bx):=\delta^{t}|\BB|\BB^{-1}, & \mbox{for} \quad \Bx\in D_\delta\setminus\overline{D}_{\delta/2},
\end{array}
\eeq
where $\BB(\Bx):=\nabla_{\mathbf{x}} A(\Bx)$ is the Jacobian matrix of the blowup transformation $A$. It is obvious that
\beq\label{eq:Bprop51}
\BB=\frac{\mathbf{e}_3\mathbf{e}_3^T}{\delta}+
\mathbf{e}_1\mathbf{e}_1^T+\mathbf{e}_2\mathbf{e}_2^T \quad \mbox{in} \quad D_\delta^0,
\eeq
is a constant matrix and
\beq\label{eq:Bprop52}
\BB \mathbf{n}= \frac{\mathbf{n}}{\delta}, \quad \BB \mathbf{e}_1= \mathbf{e}_1, \quad \BB \mathbf{e}_2=\mathbf{e}_2 \quad \mbox{in}
\quad D_\delta^0.
\eeq
Thus one has
$$
|\BB|=1/\delta, \quad \mbox{in} \quad D_\delta^0.
$$
Furthermore, for $\Bx\in D_\delta$, one can show by direct calculations that \eqref{eq:le22} and \eqnref{eq:eigenB1} are also valid for the $\BB$ defined above.
\begin{lem}\label{le:changerela51}
Let $D^j$, $j=0, 1, 2$ be defined at the beginning of this section.
Let $\bV$, $\BW$ and $\hat{\bV}$, $\hat{\BW}$ be defined similarly as those in Lemma \ref{le:changerela1}. Then one has
\begin{align*}
\int_{\p D^j}(\nu_{\tdx}\times \ti\bV(\tdx)) \cdot \BW(\tdx)\, d\Gs_{\ti x} =
& \delta^{-j} \int_{\p D^j}(\nu_{\tdx}\times\hat\bV(\tdx)) \cdot \hat\BW(\tdx)\, d\Gs_{\ti x} , \quad j=0, 1, 2. \\
\end{align*}
\end{lem}

We are now in a position to present the main theorem of this section in quantifying our partial-cloaking construction.

\begin{thm}\label{th:main2}
Let $D_\delta$ be defined in \eqref{eq:vtq} with its boundary given by \eqnref{eq:struc51}.
Let $(\bE_\delta,\bH_\delta)$ be the pair of solutions to \eqnref{eq:sys1} with
$\{\Omega; \varepsilon_\delta, \mu_\delta, \Gs_\delta\}\subset \{\RR^3; \varepsilon_\delta, \mu_\delta, \Gs_\delta\} $
defined in \eqnref{eq:struc} and $\{D_\delta\backslash\overline{D}_{\delta/2}; \varepsilon_\delta, \mu_\delta, \Gs_\delta\}$ given in \eqnref{eq:loss2}. Let $\mathbf{A}_\infty^\delta(\hat{\Bx}; \mathbf{p}, \mathbf{d})$ be the scattering amplitude of $\bE_\delta$. Define
$$
\beta_j=\min\{1, -j+r+s, -j+t+s\},\ j=0, 1, 2,
$$
and for $\epsilon\in\mathbb{R}_+$ with $\epsilon\ll 1$,
\beq\label{eq:partialmaincond51}
\Sigma_p:=\{\mathbf{p}\in\mathbb{S}^2;\ \|\mathbf{p}\times \mathbf{n}\| \leq \epsilon\},\qquad \Sigma_d:=\{\mathbf{d}\in\mathbb{S}^2;\ |\mathbf{d}\cdot\mathbf{n}| \leq \epsilon\}.
\eeq
If $r$, $s$ and $t$ are chosen such that $\beta_2-t/2> 0$, then there exists $\delta_0\in\mathbb{R}_+$ such that when $\delta<\delta_0$,
\beq\label{eq:fnles51}
\|\mathbf{A}^{\delta}_\infty(\hat{\Bx};\mathbf{p},\mathbf{d})\|\leq C (\epsilon + \delta^{2(\beta_2-t/2)}+ \delta),\quad\mathbf{p}\in\Sigma_p,\ \mathbf{d}\in\Sigma_d,\ \hat{\mathbf{x}}\in\mathbb{S}^2,
\eeq
where $C$ is a positive constant depending on $\omega$ and $D$,
but independent of $\varepsilon_a$, $\mu_a$, $\Gs_a$ and $\hat{\Bx}$, $\mathbf{p}$, $\mathbf{d}$.
\end{thm}

\begin{rem}
Similar to Remark~\ref{lem:physical1}, one readily has from Theorems~\ref{thm:transop} and \ref{th:main2} that the push-forwarded structure in \eqnref{eq:clkstruc1} produces an approximate partial-cloaking device which is capable of nearly cloaking an arbitrary EM content.
The highest accuracy of such a construction can be achieved, say for example, by taking $r=t=0$ and $s=5/2$.
\end{rem}
\begin{rem}
Since $\mathbf{p}\cdot\mathbf{d}=0$ (cf. \eqref{eq:pdo}), one has
\begin{equation}\label{eq:relationn1}
(\mathbf{d}\cdot\mathbf{n})\mathbf{p}=\mathbf{d}\times(\mathbf{p}\times\mathbf{n}),
\end{equation}
from which one can infer that the definition of $\Sigma_d$ in \eqref{eq:partialmaincond51} is redundant. However, we specified it for clarity and definiteness.
\end{rem}
%
%
%
%
\subsection{Auxiliary lemmas and asymptotic expansions}

We shall follow a similar strategy of the proof for Theorem~\ref{th:main1} in proving Theorem~\ref{th:main2}. In what follows, we adopt similar notations as those in Section~\ref{sect:4}. If we let $\Bz$ denote the space variable on $\GG_0$, then for any $\By\in \p D_q$, we define $\mathbf{z}_{y}$ to be the projection of $\By$ onto $\GG_0$. We also define $\tilde\ba (\tdy):=\ba (\By)$ and $\tdy:= A(\By)\in \overline{D}$ for $\By\in \overline{D}_\delta$.

The following lemma is a counterpart to Lemma \ref{le:1} in Section~\ref{sect:4}.
\begin{lem}\label{le:5.1}
Let $D_\delta$ be described in \eqref{eq:vtq} and \eqref{eq:struc51}. Define
\beq\label{eq:Msca51}
\Mcal_{S^0}^{\omega}[\tilde\ba ](\tdx):=\mathrm{p.v.}\quad \mathbf{\nu}_{\ti{\mathbf{x}}}\times \nabla_{\mathbf{z}_{\ti{x}}}\times \int_{S^0}G_\omega(\Bz_{\ti{x}}-\Bz_{\ti{y}})\ti\ba (\tdy)\, d\Gs_{\ti{y}},
\eeq
and
$$
\Mcal_{S^{2}}[\tilde\ba ](\tdx):=\mathrm{p.v.}\quad \mathbf{\nu}_{\ti{\mathbf{x}}}\times \nabla_{\ti{\mathbf{x}}}\times \int_{S^{2}}G_\omega(\tdx-\tdy)\ti\ba (\tdy)\, d\Gs_{\ti{y}}.
$$
Define
$$\Lcal_{D_\delta}^{\omega}[\ba](\Bx):=\nu_{\mathbf{x}}\times\nabla_{\mathbf{x}}\times\nabla_{\mathbf{x}}
\times \Acal_{D_\delta}^{\omega}[\ba](\Bx), $$
and
$$\Lcal_{S^0}^{\omega}[\ti\ba](\Bx):=\nu_{\ti{\mathbf{x}}}\times\nabla_{\Bz_{\ti x}}\times\nabla_{\mathbf{z}_{\ti x}}
\times \int_{S^0}G_\omega(\Bz_{\ti{x}}-\Bz_{\ti{y}})\ti\ba (\tdy)\, d\Gs_{\ti{y}}.$$
Then one has
\beq\label{eq:arg51}
\Mcal_{D_\delta}^\omega[\ba ](\Bx)=
\left\{
\begin{array}{ll}
\Mcal_{S^0}^\omega[\tilde\ba ](\tdx)+\Ocal(\delta\|\tilde\ba \|_{\mathrm{TH}^{-1/2}(\p D)}),
& \Bx\in S_\delta^0 \cup S_\delta^1,  \\
\Mcal_{S^0}^\omega[\tilde\ba ](\tdx)+\Mcal_{S^{2}}[\tilde\ba ](\tdx) +\Ocal(\delta\|\tilde\ba \|_{\mathrm{TH}^{-1/2}(\p D)}),
& \Bx\in S_\delta^2,
\end{array}
\right.
\eeq
and
\beq\label{eq:arg52}
\Lcal_{D_\delta}^{\omega}[\ba](\Bx)=
\left\{
\begin{array}{ll}
\Lcal_{S^0}^\omega[\tilde\ba ](\tdx)+ \Ocal(\delta\|\tilde\ba \|_{\mathrm{TH}^{-1/2}(\p D)}),
& \Bx\in S_\delta^0 \cup S_\delta^1, \\
\frac{1}{\delta}\nu_{\ti{\mathbf{x}}}\times\nabla_{\ti{\mathbf{x}}}\times\nabla_{\ti{\mathbf{x}}}\times\Acal_{S^{2}}[ \tilde\ba ](\tdx) +\Ocal(\|\tilde\ba \|_{\mathrm{TH}^{-1/2}(\p D)}),
& \Bx\in S_\delta^2.
\end{array}
\right.
\eeq
\end{lem}

\begin{proof}
By following a similar argument to that in the proof of Lemma \ref{le:1}, one can compute for $\Bx\in S_\delta^0 \cup S_\delta^1 $ that
$$
\|\Bx-\By\|=\|\Bx-\mathbf{z}_x-(\By-\mathbf{z}_y)+ \mathbf{z}_x-\mathbf{z}_y\|=\|\mathbf{z}_{\tilde{x}}-\mathbf{z}_{\tilde{y}}\|+ \Ocal(\delta),
$$
and
$$
\la \Bx-\By, \mathbf{\nu}_{\mathbf{x}}\ra= \la\mathbf{z}_x-\mathbf{z}_y, \mathbf{\nu}_{\tilde{\mathbf{x}}}\ra+\Ocal(\delta),
\quad e^{i\omega|\Bx-\By|}= e^{i\omega|\mathbf{z}_{\tilde{x}}-\mathbf{z}_{\tilde{y}}|}+\Ocal(\delta).
$$
One also notes that for $\Bx\in S_\delta^2$
$$
\|\Bx-\By\|=\delta\|\tdx-\tdy\|, \quad e^{i\omega\|\Bx-\By\|}=1+ \Ocal(\delta).
$$
By using the above facts, the proof can then be completed by using the similar expansion method as that in the proof of Lemma \ref{le:1}.
\end{proof}

Next, by straightforward calculations, one can show that for $D_\delta$ described in \eqref{eq:vtq} and \eqref{eq:struc51}, there holds the following far field expansion for $\Bx\in \RR^3\setminus \overline{D}_\delta$,
\begin{align}
\Scal_{D_\delta}^\omega[\ba ](\Bx)= & \int_{\p D_\delta} G_\omega(\Bx-\By) \ba (\By) d\Gs_y =  \int_{S^0} G_\omega(\Bx-\mathbf{z}_{\tilde{y}})\tilde\ba (\tdy)d\Gs_{\tilde{y}}\nonumber\\
& +\delta \int_{S^1} G_\omega(\Bx-\mathbf{z}_{\tilde{y}}) \tilde\ba (\tdy)d\Gs_{\tilde{y}}
+ \delta \int_{S^0} (\tdy-\mathbf{z}_{\tilde{y}})\cdot\nabla G_\omega(\Bx-\mathbf{z}_{\tilde{y}})\tilde\ba (\tdy)d\Gs_{\tilde{y}}\nonumber \\
& + \Ocal\Big(\delta^2 \|\Bx\|^{-1}\|\tilde\ba \|_{\mathrm{TH}^{-1/2}(\p D)}\Big).\label{eq:expan51}
\end{align}
With the above expansion, we can show the following theorem
\begin{thm}\label{th:impotestm51}
Let $\bE_\delta$ be the solution to \eqnref{eq:sys1} with $D_\delta$ described in \eqref{eq:vtq} and \eqref{eq:struc51}. Define
$\tilde\Phi(\tdy):=\Phi(\By)=\mathbf{\nu}_{\mathbf{y}} \times  \bE_\delta(\By)\Big|_{\partial D_\delta}^+$,
then there holds for $\Bx\in\mathbb{R}^3\backslash\overline{D}$
\begin{align}
(\bE_\delta-\bE^i)(\Bx)= \int_{\GG_0}G_\omega(\Bx-\mathbf{z}_{\tilde{y}})\mathbb{M}[\mathbf{n}\times \bE^i(\mathbf{z})](\Bz_{\tilde y})\, d\Gs_{\mathbf{z}_{\tilde{y}}} \nonumber \\
+\Ocal\left(\|\Bx\|^{-1}\left(\|\ti\Phi\|_{\mathrm{TH}^{-1/2}(S^0)}+\delta(\|\ti\Phi\|_{\mathrm{TH}^{-1/2}(\p D)}+1)\right)\right),\label{eq:expan53}
\end{align}
where
\beq\label{eq:op511}
\mathbb{M}:=\left(-\frac{I}{4}+\Mcal_{\GG_0}^\omega\right)^{-1}\Mcal_{\GG_0}^\omega
\left(\frac{I}{4}+\Mcal_{\GG_0}^\omega\right)^{-1}
\eeq
with $\Mcal_{\GG_0}^\omega$ defined by
\beq\label{eq:op512}
\Mcal_{\GG_0}^\omega[\Theta(\Bz)](\Bz_{\ti x}):=\mathbf{n}\times \nabla_{\mathbf{z}_{\ti x}}\times \int_{\GG_0}G_{\omega}(\Bz_{\ti x}
-\Bz_{\ti y})\Theta(\Bz_{\ti y})\, d\Gs_{\Bz_{\ti y}}.
\eeq
\end{thm}

\begin{proof}
We first recall that the solution to \eqnref{eq:sys1} can be represented by \eqnref{eq:vis1} and \eqnref{eq:vis2},
where the density function $\ba $ satisfies \eqnref{eq:trans1}.
Next, by using the fact that
$$
-\frac{I}{2}+\Mcal_{S^0}^\omega : \mathrm{TH}^{-1/2}(S^0)\rightarrow \mathrm{TH}^{-1/2}(S^0)
$$
is invertible (cf. \cite{nedelec,Tay}), and using \eqnref{eq:trans1}, \eqnref{eq:arg51}, along with the use of the expansion of $\bE^i$ in $\Bz$, one has by direct calculations that
\begin{align}
\tilde{\ba }(\tdy)= &\Big(-\frac{I}{2}+\Mcal_{S^0}^\omega\Big)^{-1}\left[\nu\times\bE^i(\Bz)+\ti\Phi\right](\tdy) + \Ocal(\delta(\|\ti\Phi\|_{\mathrm{TH}^{-1/2}(\p D)}+1))\nonumber \\
= &\Big(-\frac{I}{2}+\Mcal_{S^0}^\omega\Big)^{-1}\left[\nu\times\bE^i(\Bz)\right](\tdy)\nonumber \\
& + \Ocal\left(\|\ti\Phi\|_{\mathrm{TH}^{-1/2}(S^0)}+\delta(\|\ti\Phi\|_{\mathrm{TH}^{-1/2}(\p D)}+1)\right)\label{eq:111}
\end{align}
for all $\tdy \in S^0 \cup S^1$.
Noting that for $\tdy\in \GG^1$, $(\tdy-2\mathbf{n})\in \GG^2$, we define
$\tilde\psi(\tdy):=\tilde\ba (\tdy-2\mathbf{n})$ for $\tdy\in \GG^1$ and
$\tilde\psi(\tdy):=\tilde\ba (\tdy+2\mathbf{n})$ for $\tdy\in \GG^2$. By using the fact that
$$
\mathbf{\nu}_{\tilde{\mathbf{y}}-2n_{\tilde{\mathbf{y}}}}= - \mathbf{\nu}_{\tilde{\mathbf{y}}}=-\mathbf{n} \quad \mbox{for} \quad \tdy\in \GG^1,
$$
and the definition \eqnref{eq:Msca51}, one obtains (assuming for a while that $\tdx\in S^0$)
\begin{align}
\tilde{\ba }(\tdy-2\mathbf{n})=\tilde\psi(\tdy)
= &\Big(\frac{I}{2}+\Mcal_{S^0}^\omega\Big)^{-1}\Big[\mathbf{n}\times\bE^i(\Bz_{\ti x})\Big](\tdy) \nonumber \\
& + \Ocal\left(\|\Bx\|^{-1}\left(\|\ti\Phi\|_{\mathrm{TH}^{-1/2}(S^0)}+\delta(\|\ti\Phi\|_{\mathrm{TH}^{-1/2}(\p D)}+1)\right)\right)\label{eq:expan522}
\end{align}
for all $\tdy\in \GG^1$.
By inserting \eqnref{eq:111} and \eqnref{eq:expan522} into \eqnref{eq:expan51}, one then has that for $\Bx\in \RR^3\setminus \overline{D}$,
\begin{align*}
&\Scal_{D_\delta}^\omega[\ba ](\Bx)=\int_{S^0} G_\omega(\Bx-\mathbf{z}_{\tilde{y}})\tilde\ba (\tdy)\, d\Gs_{\tilde{y}}+
\Ocal\left(\|\Bx\|^{-1}\left(\delta(\|\ti\Phi\|_{\mathrm{TH}^{-1/2}(\p D)}+1)\right)\right) \\
=& \int_{\GG^1} G_\omega(\Bx-\mathbf{z}_{\tilde{y}})\big(\tilde\ba (\tdy)+\tilde\psi(\tdy)\big)\, d\Gs_{\tilde{y}}+
+ \Ocal\left(\|\Bx\|^{-1}\left(\delta(\|\ti\Phi\|_{\mathrm{TH}^{-1/2}(\p D)}+1)\right)\right)\\
=&\int_{\GG^1}G_\omega(\Bx-\mathbf{z}_{\tilde{y}})\hat{\mathbb{M}}[\mathbf{n}\times \bE^i(\mathbf{z})](\tdy)\, d\Gs_{\ti y}\\
&+\Ocal\left(\|\Bx\|^{-1}\left(\|\ti\Phi\|_{\mathrm{TH}^{-1/2}(S^0)}+\delta(\|\ti\Phi\|_{\mathrm{TH}^{-1/2}(\p D)}+1)\right)\right),
\end{align*}
where $\hat{\mathbb{M}}$ is defined by
\beq\label{eq:op51}
\hat{\mathbb{M}}:=2\left(-\frac{I}{2}+\Mcal_{S^0}^\omega\right)^{-1}\Mcal_{S^0}^\omega
\left(\frac{I}{2}+\Mcal_{S^0}^\omega\right)^{-1}.
\eeq
Since the function $\Theta(\Bz):=\mathbf{n}\times\bE^i(\mathbf{z})$ depends only on $\Bz\in \GG_0$, and hence by
definition \eqnref{eq:op512} one readily verifies that
$$
\Mcal_{\GG_0}^\omega[\Theta](\Bz_{\ti x})=\frac{1}{2}\Mcal_{S^0}^\omega[\hat\Theta(\tdy)](\Bz_{\ti x} +\mathbf{n}),
\quad \tdx \in \GG_1,
$$
for $\hat\Theta(\tdy)=\mathbf{n}\times\bE^i(\mathbf{z}_{\tilde{y}})$, $\tdy\in S^0$.
Therefore, we can replace the operator $\hat{\mathbb{M}}$ in \eqnref{eq:op51} to be
$$
\mathbb{M}:=\left(-\frac{I}{4}+\Mcal_{\GG_0}^\omega\right)^{-1}\Mcal_{\GG_0}^\omega
\left(\frac{I}{4}+\Mcal_{\GG_0}^\omega\right)^{-1},
$$
and the proof is complete.
\end{proof}

By using Theorem \ref{th:impotestm51}, we next consider a particular case by assuming that $\tilde\Phi\equiv 0$. Physically speaking, this corresponds to the case that $D_\delta$ is a so-called perfectly electric conductor. The result in the next theorem already partly reveals the partial-cloaking effect.

\begin{thm}\label{th:pc1}
Let $D_\delta$ be as described in \eqref{eq:vtq} and \eqref{eq:struc51}. Consider the following scattering problem
\begin{equation}\label{eq:syspec51}
\left\{
\begin{array}{ll}
 \nabla \times \bE_\delta^p-i \omega \bH_\delta^p=0 &  \mbox{in} \quad \RR^3\setminus \overline{D}_\delta, \\
 \nabla  \times \bH_\delta^p+i \omega \bE_\delta^p=0 & \mbox{in} \quad \RR^3\setminus \overline{D}_\delta, \\
 \nu \times \bE_\delta^p \big|_+  = 0 &  \mbox{on} \quad \p D_\delta,
\end{array}
\right.
\end{equation}
subject to the Silver-M\"{u}ller radiation condition:
\beq\label{eq:radia1}
\lim_{|\Bx|\rightarrow\infty} \|\Bx\| ( (\bH_\delta^p- \bH^i) \times\hat{\Bx}- (\bE_\delta^p-\bE^i))=0.
\eeq
Let $\bE^\delta_\infty(\hat{\Bx}; \mathbf{E}^i)$ be the corresponding scattering amplitude to \eqref{eq:syspec51}--\eqref{eq:radia1}. Then there holds
\beq\label{eq:pcscatamp1}
\Big\|\bE^\delta_\infty(\hat{\Bx}; \bE^i)+\frac{1}{2\pi}\int_{\GG_0}e^{-i\omega\frac{4\pi}{3}\sum\limits_{m=-1}^1 Y_1^m
(\hat{\Bx})\overline{Y_1^m(\hat{\Bz}_{\tilde{y}})}|\Bz_{\tilde{y}}|}\mathbb{M}[\mathbf{n}\times \bE^i(\mathbf{z})](\Bz_{\tilde y})\, d\Gs_{z_{\tilde{y}}}\Big\|\leq C\delta,
\eeq
where $\mathbb{M}$ is defined in \eqref{eq:op511} and $C$ depends only on $\omega$ and $D$.
Furthermore, if there holds
\beq\label{eq:patialcon51}
\|\mathbf{n}\times \mathbf{p}\|\leq \epsilon, \quad \epsilon\ll 1,
\eeq
then for sufficient small $\delta\in \RR_+$, one has
\beq\label{eq:patialres51}
\|\bE^\delta_\infty(\hat{\Bx}; \bE^i)\|\leq C (\epsilon+\delta),
\eeq
where $C$ depends only on $\omega$ and $D$.
\end{thm}

\begin{proof}
We start with the following addition formula (see, e.g., \cite{nedelec})
\begin{align}
\frac{1}{\|\Bx-\By\|}= 4\pi\sum_{n=0}^\infty \sum_{m=-n}^{n} \frac{1}{2n+1}Y_n^m(\xi, \vartheta) \overline{Y_n^m(\xi', \vartheta')} \,\frac{q'^n}{q^{n+1}},\label{Gammaexp1}
\end{align}
where $(q,\xi, \vartheta)$ and $(q',\xi', \vartheta')$ are the
spherical coordinates of $\Bx$ and $\By$, respectively; and $Y_n^m$ is the spherical harmonic function of degree $n$ and order $m$.
For simplicity, the parameters $(q, \xi, \vartheta)$ and $(q', \xi', \vartheta')$ shall be replaced by $(\|\Bx\|, \hat{\Bx})$
and $(\|\By\|, \hat{\By})$, respectively. It then follows by \eqnref{Gammaexp1} that
\begin{equation}\label{eq:333}
\|\Bx-\By\|=\|\Bx\|\left(1-\frac{4\pi}{3}\sum_{m=-1}^{1}Y_1^m(\hat{\Bx})\overline{Y_1^m(\hat{\By})}\frac{\|\By\|}{\|\Bx\|}\right)
+\Ocal(\|\Bx\|^{-1}).
\end{equation}
The solution $\bE_\delta^p$ in \eqnref{eq:syspec51} and \eqnref{eq:radia1} can be represented by
$$
\bE_\delta^p=\bE^i + \Acal_{D_\delta}^{\omega}[\ba] \quad \mbox{in} \quad \RR^3\setminus D_\delta,
$$
with $\ba$ satisfying
$$
\Big(-\frac{I}{2}+\Mcal_{D_\delta}^{\omega}\Big)[\ba](\By)=-\mathbf\nu_{\mathbf{y}}\times \bE^i(\By), \quad \By\in \p D_\delta.
$$
Then one can easily derive from the expansion formula \eqnref{eq:expan53} with $\ti\Phi=0$ that
\begin{equation}\label{eq:ffppnn1}
(\bE_\delta-\bE^i)(\Bx)= 2\int_{\GG_0}G_\omega(\Bx-\mathbf{z}_{\tilde{y}})\mathbb{M}[\mathbf{n}\times \bE^i(\mathbf{z})](\Bz_{\tilde y})\, d\Gs_{z_{\tilde{y}}}+\Ocal(\|\Bx\|^{-1}\delta), \ \Bx\in \RR^3\setminus D_\delta,
\end{equation}
where $\mathbb{M}$ is defined by \eqnref{eq:op511}.
By combining \eqnref{eq:ffppnn1} and \eqref{eq:333}, we then have
\begin{equation}\label{eq:444}
\begin{split}
& \bE_\infty^\delta(\hat{\Bx}; \bE^i)=-\frac{1}{2\pi}\int_{\GG_0}e^{-i\omega\frac{4\pi}{3}\sum\limits_{m=-1}^1 Y_1^m
(\hat{\Bx})\overline{Y_1^m(\hat{\Bz}_{\tilde{y}})}|\Bz_{\tilde{y}}|}\mathbb{M}[\mathbf{n}\times \bE^i(\mathbf{z})](\Bz_{\tilde y})\, d\Gs_{z_{\tilde{y}}}+\Ocal(\delta),
\end{split}
\end{equation}
which readily proves \eqnref{eq:pcscatamp1}.
Next by the definition of $\mathbb{M}$ in \eqnref{eq:op511}, one can show that
$$
\|\mathbb{M}[\mathbf{n}\times\bE^i(\mathbf{z})]\|_{\mathrm{TH}^{-1/2}(\GG_0)}\leq C \|\mathbf{n}\times\bE^i(\mathbf{z})\|_{\mathrm{TH}^{-1/2}(\GG_0)},
$$
where $C$ depends only on $\GG_0$ and $\omega$. This together with \eqnref{eq:444} further implies that
\begin{equation}\label{eq:ffppnn2}
\|\bE_\infty^\delta(\hat{\Bx}; \bE^i)\|\leq C(\|\mathbf{n}\times\bE^i(\mathbf{z})\|_{\mathrm{TH}^{-1/2}(\GG_0)}+\delta).
\end{equation}
Finally, by inserting the condition \eqref{eq:patialcon51} into \eqref{eq:ffppnn2} and direct calculations, one can easily show \eqnref{eq:patialres51}.

The proof is complete.
\end{proof}

\subsection{Proof of Theorem~\ref{th:main2}}

In this section, we present the proof of Theorem~\ref{th:main2}, which follows a similar spirit to those of Theorems~\ref{th:main1} and \ref{th:pc1}. The major idea is to control the norm $\|\tilde\Phi\|_{\mathrm{TH}^{-1/2}(\partial D)}$, which was taken to be identically zero in Theorem~\ref{th:pc1}.

\begin{lem}\label{le:mainpf51}
Let $(\bE_\delta,\bH_\delta)$ be the pair of solutions to \eqnref{eq:sys1} with
$\{\Omega; \varepsilon_\delta, \mu_\delta, \Gs_\delta\}\subset \{\RR^3; \varepsilon_\delta, \mu_\delta, \Gs_\delta\} $
defined in \eqnref{eq:struc} and $\{D_\delta\backslash\overline{D}_{\delta/2}; \varepsilon_\delta, \mu_\delta, \Gs_\delta\}$ given in \eqnref{eq:loss2}.
Then there holds the following estimates for $j=0, 1, 2$,
\begin{align*}
\int_{D_\delta^j\setminus D_{\delta/2}} \|\bE_\delta \|^2 d\Bx \leq &
C \delta^{j-t}\Big(\|\nu\times \bE_\delta^+\|_{\mathrm{TH}_{\mathrm{div}}^{-1/2}(\p B_R)}^2
+ \|\ti\ba\|_{\mathrm{TH}^{-1/2}(S^0)}+\delta\|\ti\ba\|_{\mathrm{TH}^{-1/2}(\p D)}\Big),
\end{align*}
where the constant $C$ depends only on $R$ and $\omega$.
\end{lem}
\begin{proof}
We first note that \eqnref{eq:estL1} still holds for the scattering problem described in the present lemma. Then one has
\begin{align*}
\int_{D_\delta\setminus\overline{D}_{\delta/2}} \Gs_l \bE_\delta\cdot \overline{\bE}_\delta d\Bx \leq
C\|\nu\times \bE_\delta^+\|_{\mathrm{TH}_{\mathrm{div}}^{-1/2}(\p B_R)}\|\GL(\nu\times \bE_\delta^+)\|_{\mathrm{TH}_{\mathrm{div}}^{-1/2}(\p B_R)}
+\mathbb{R}_1 + \mathbb{R}_2,
\end{align*}
where by using \eqnref{eq:vis2} and \eqnref{eq:arg52} one further has
\begin{align*}
\mathbb{R}_1&=\Big|\Re\int_{\p D_\delta}(\nu\times\overline{\bE^i})\cdot(\nu\times \nu\times \bH_\delta^+)\Big|_+ d\Gs_x\Big|
=\Big|\Re\int_{\p D_\delta}\overline{\bE^i}\cdot(\nu\times \bH_\delta^+)\Big|_+\, d\Gs_x\Big|\\
&\leq \frac{1}{\omega}\Big|\int_{S^0}\overline{\bE^i}(\Bz_{\ti x})\cdot\Lcal_{S^0}^\omega[\tilde\ba ](\tdx)\, d\Gs_{\ti x}\Big| +\Ocal(\delta\|\ti\ba\|_{\mathrm{TH}^{-1/2}(\p D)}),
\end{align*}
and by using \eqnref{eq:arg51} one further has
\begin{align*}
\mathbb{R}_2&=\Big|\Re\int_{\p D_\delta} (\nu\times\overline{\bE_\delta^+})\Big|_+\cdot(\nu\times\nu\times \bH^i)\, d\Gs_x\Big|
=\Big|\Re\int_{\p D_\delta}\overline{\bE_\delta^+}\Big|_+\cdot(\nu\times \bH^i)\, d\Gs_x\Big|\\
&\leq \delta\Big|\int_{S^0}\Mcal_{S^0}^\omega[\tilde\ba ](\tdx)\cdot  \bH^i(\Bz_{\ti x})\, d\Gs_{\ti x}\Big|
+\Ocal(\delta^2\|\ti\ba\|_{\mathrm{TH}^{-1/2}(\p D)}).
\end{align*}
Combining the above estimates with \eqnref{eq:Bprop51}, \eqnref{eq:Bprop52} and the definition of $\Gs_l$ in \eqnref{eq:loss2}, we can complete the proof.
\end{proof}

\begin{proof}[Proof of Theorem \ref{th:main2}]
First, by using \eqnref{eq:expan53} and \eqnref{eq:partialmaincond51} one obtains
\beq\label{eq:est51}
\|\mathbf{A}_\infty^{\delta}(\hat{\Bx}; \mathbf{p},\mathbf{d})\|\leq C\Big(\epsilon+ \|\tilde{\Phi}\|_{\mathrm{TH}^{-1/2}(S^0)}
+ \delta(\|\tilde{\Phi}\|_{\mathrm{TH}^{-1/2}(\partial D)}+1)\Big),
\eeq
where $\tilde\Phi(\tdx)=\Phi(\Bx):=\nu\times\bE_\delta\Big|_+(\Bx)$ for $\Bx\in \partial D_\delta$.
The following estimate can be obtained by using the result in Lemma \ref{le:changerela51} and a completely similar argument as that in the proof of Lemma \ref{le:estiphi41},
\beq\label{eq:tmp51}
\begin{split}
&\|\tilde{\Phi}\|_{\mathrm{TH}^{-1/2}(S^j)}\leq C\delta^{-(j+1)/2+\beta_j}\|\bE_\delta\|_{L^2(D_\delta^j\setminus D_{\delta/2})^3},\quad j=0, 1, 2,\\
\end{split}
\eeq
where $\beta_j=\min\{1, -j+r+s, -j+t+s\}$, $j=0, 1, 2$.
On the other hand, by using \eqnref{eq:111}, Lemma \ref{le:mainpf51} and \eqnref{eq:est51}, one has
\begin{align}
\int_{D_\delta^j\setminus D_{\delta/2}} \|\bE_\delta\|^2 d\Bx \leq &
C \delta^{j-t}\Big( \|\tilde{\Phi}\|_{\mathrm{TH}^{-1/2}(S^0)}^2
+ \delta^2\|\tilde{\Phi}\|_{\mathrm{TH}^{-1/2}(\partial D)}^2 \label{eq:esteng51}\\
&\quad \quad
+ \|\ti\Phi\|_{\mathrm{TH}^{-1/2}(S^0)}+\delta\|\ti\Phi\|_{\mathrm{TH}^{-1/2}(\p D)}+\epsilon+\delta\Big), \quad j=0, 1, 2.\nonumber
\end{align}
Inserting \eqnref{eq:esteng51} back into \eqnref{eq:tmp51}, one can show
\begin{align}
\|\tilde{\Phi}\|_{\mathrm{TH}^{-1/2}(S^j)}\leq & C \delta^{\beta_j-t/2-1/2}\Big( \|\tilde{\Phi}\|_{\mathrm{TH}^{-1/2}(S^0)}
+\delta^{1/2} \|\ti\Phi\|_{\mathrm{TH}^{-1/2}(\p D)} \nonumber \\
&+ \|\ti\Phi\|_{\mathrm{TH}^{-1/2}(S^0)}^{1/2}+\epsilon^{1/2}+\delta^{1/2}\Big), \quad j=0, 1, 2 .\label{eq:farfieldest51}
\end{align}
Noting that $\beta_2\leq \beta_1\leq \beta_0$ and $\beta_2-t/2>0$, one has for $\delta\in \RR_+$ sufficiently small that
\beq\label{eq:farfieldest52}
\|\tilde{\Phi}\|_{\mathrm{TH}^{-1/2}(\p D)}\leq C \delta^{\beta_2-t/2-1/2}\Big( \|\tilde{\Phi}\|_{\mathrm{TH}^{-1/2}(S^0)}
+ \|\ti\Phi\|_{\mathrm{TH}^{-1/2}(S^0)}^{1/2}+\epsilon^{1/2}+\delta^{1/2}\Big).
\eeq
Inserting \eqnref{eq:farfieldest52} back into \eqnref{eq:farfieldest51} (for $j=0$),
there holds
\begin{equation}\label{eq:ffppnn3}
\|\tilde{\Phi}\|_{\mathrm{TH}^{-1/2}(S^0)}\leq C \delta^{\beta_0-t/2-1/2}\Big( \|\tilde{\Phi}\|_{\mathrm{TH}^{-1/2}(S^0)}
+ \|\ti\Phi\|_{\mathrm{TH}^{-1/2}(S^0)}^{1/2}+\epsilon^{1/2}+\delta^{1/2}\Big).
\end{equation}
By $\beta_2-t/2>0$, it is directly verified that $\beta_0-t/2\geq 1$ and hence
\begin{equation}\label{eq:ffppnn4}
2(\beta_0-t/2)-1\geq \beta_0-t/2.
\end{equation}
Using \eqref{eq:ffppnn3} and \eqref{eq:ffppnn4}, one readily infers for $\delta\in \RR_+$ sufficiently small that
\beq\label{eq:farfieldest53}
\|\tilde{\Phi}\|_{\mathrm{TH}^{-1/2}(S^0)}\leq C (\delta^{\beta_0-t/2} + \epsilon).
\eeq
Now by inserting \eqnref{eq:farfieldest53} into \eqnref{eq:farfieldest52}, one has
\beq\label{eq:farfieldest54}
\|\tilde{\Phi}\|_{\mathrm{TH}^{-1/2}(\p D)}\leq C \delta^{\beta_2-t/2-1/2}(\delta^{1/2} + \epsilon^{1/2})
\leq C(\delta^{\beta_2-t/2}+\delta^{2(\beta_2-t/2)-1}+\epsilon).
\eeq
Finally by plugging \eqnref{eq:farfieldest53} and \eqnref{eq:farfieldest54} into \eqnref{eq:est51},
we arrive at \eqnref{eq:fnles51}.

The proof is complete.
\end{proof}

\section*{Acknowledgment}

The work of Y. Deng was partially supported by the Mathematics and Interdisciplinary Sciences Project, Central South University, China. The work of H. Liu was partially supported by Hong Kong RGC General Research Funds, HKBU 12302415 and 405513, and the NSF grant of China, No. 11371115. The work of G. Uhlmann was supported by NSF and the Academy of Finland.

\end{document}